\documentclass[12pt]{tran-l}

\newtheorem{thm}{Theorem}[subsection]
\newtheorem{cor}[thm]{Corollary}
\newtheorem{lem}[thm]{Lemma}

\theoremstyle{definition}

\theoremstyle{remark}
\newtheorem{rem}[thm]{Remark}

\usepackage[utf8]{inputenc}   
\usepackage[T1]{fontenc}      
\usepackage[english]{babel}   
\usepackage[a4paper, left=3cm, right=2.5cm, top=2.5cm, bottom=2.5cm]{geometry}

\numberwithin{equation}{subsection}

\begin{document}

\title[Binary matrices associated with the Pell sequence.]
 {Binary matrices of order 3 associated with the Pell sequence}

\author{Wilson Arley Martinez * , Samin Ingrith Ceron}

\address{Martinez, W.A.; Departmento de Matem\'aticas, Universidad del Cauca , Popay\'an , Colombia}
\email{wamartinez@unicauca.edu.co}

\address{Ceron, S.I.; Departmento de Matem\'aticas, Universidad del Cauca , Popay\'an , Colombia}
\email{sicbravo@gmail.com}

\thanks{This work was completed with the support of the Universidad del Cauca.}

\thanks{The author was also supported by the research group “Estructuras Algebraicas, Divulgaci\'on Matem\'atica y Teorías Asociadas. @DiTa”.}

\thanks{* Corresponding Author: Wilson Arley Martinez.}




\date{September 18, 2025; revised September 30, 2025}

\dedicatory{}

\commby{W.A.M}


\begin{abstract}

In this paper, we construct Pell matrices, analogous to Fibonacci matrices, to study algebraic properties of Pell numbers via linear algebra. This framework yields identities involving the trace, inverse, and determinant, as well as matrix products that generate recurrence relations and closed-form expressions. Additionally, we classify all binary $3 \times 3$ matrices that generate the Pell equation through conjugation, providing a complete characterization of such matrices.
\end{abstract}

\maketitle

\section*{Introduction}

The \emph{Pell sequence} $E_n$ is a classical sequence of integers defined recursively by the linear recurrence relation
\begin{equation*}
    E_{n+1} = 2E_n + E_{n-1}, \quad \text{for } n \ge 2, \quad \text{with initial values } E_1 = 1, \, E_2 = 2.
\end{equation*}
The first terms of the sequence are given by
\[
1, \; 2, \; 5, \; 12, \; 29, \; 70, \; 169, \; 408, \; \dots
\]
Matrix methods provide a powerful framework for studying linear recurrence relations. In particular, it is well-known (see, e.g.,\cite{Bicknell1975, Ercolano1979}) that the Pell numbers can be generated using powers of the $2 \times 2$ matrix
\[
M = 
\begin{bmatrix}
2 & 1 \\[2mm]
1 & 0
\end{bmatrix}.
\]
More precisely, for all integers $n \ge 1$, one has
\[
M^n = 
\begin{bmatrix}
E_{n+1} & E_n \\[1mm]
E_n & E_{n-1}
\end{bmatrix},
\]
which allows a compact and elegant representation of the sequence in terms of linear algebra.

\medskip 

The Pell sequence has been extensively studied, with particular attention to the greatest common divisor and divisibility properties of its terms (see, e.g., \cite{KilicTasci2005, McDaniel1991, MurtyPadhy2023}).  
For arbitrary integers $a$ and $b$, some of the most notable identities include
\begin{equation*}\label{eq:pell-identities}
\begin{aligned}
E_{n+a}E_{n+b} - E_n E_{n+a+b} &= E_a E_b (-1)^n, \\[1mm]
E_m E_{n+1} - E_{m+1} E_n &= (-1)^n E_{m-n}, \\[1mm]
E_n^2 + E_{n+1}^2 &= E_{2n+1}.
\end{aligned}
\end{equation*}
Other identities can be derived from these basic relations.

\medskip 

These identities play a central role in the analysis of sequences such as the Fibonacci and Jacobsthal numbers and their applications in number theory, combinatorics, and related areas (see, e.g., \cite{Koshy2001, KokenBozkurt2008, ThongkamSailadda2018}). Furthermore, matrix approaches provide a systematic way to derive new identities, study the algebraic structure of sequences, and explore connections with continued fractions and Diophantine equations.

\medskip 

In this work, we focus on the complete characterization of all $3 \times 3$ binary matrices associated with Pell numbers, examining their algebraic properties and deriving new identities. Moreover, we introduce the sequences $r_n$ and $b_n$ related to the Pell sequence, with $r_n$ being studied for its divisibility properties and its connections with Sidon sets. The sequences ${r_n}$ and ${b_n}$ are defined explicitly in Lemma ~\ref{l:5} and ~\ref{l:11}, where their main properties are presented.

\medskip 

Moreover, we naturally encounter a generalized Pell sequence, defined by the recurrence relation 
\[
Q_n = 2 Q_{n-1} + Q_{n-2}
\]
with initial conditions
\[
Q_0 = 1, \quad Q_1 = 3.
\]
This sequence can be viewed as a generalization of the classical Pell sequence shifted by one index (i.e., \(Q_n = q_{n+1}\) with \(q_0=1, q_1=1\); see \cite{Horadam1994}) and differs only in its initial terms.

\medskip 

In this paper, we present several identities involving Pell numbers and related matrices. Some of these identities are well-known and documented in the literature (e.g.,\cite{Serkland1972,Horadam1994,Bicknell1975,Lucas1878,Horadam1961,Koshy2014}), while others constitute original contributions of this work. For the classical identities, we provide precise references, whereas the new results are highlighted through the matrix approach and the methods developed in later sections.

\section{Recurrences Generated by Binary $3 \times 3$ Matrices with Determinant Zero}
Consider the generating matrix
\[
U = \begin{pmatrix}0 & 0 & 1\\ 1 & 1 & 1\\ 1 & 1 & 1\end{pmatrix}.
\]
In this section, we examine the Pell \textit{U}-matrix with determinant~0, which provides a matrix representation of Pell numbers. This matrix is employed to compute powers \(U^n\) and derive associated identities; to determine the characteristic roots and the Binet formulas for both the Pell sequence and its generalized form; to establish identities involving these sequences; and to obtain summation formulas for the Pell numbers.

\subsection{The Matrix Representation of Pell Numbers}
\begin{lem}\label{l:1} 
Let $u$ be the matrix
\[
u = \begin{pmatrix} 
                    0   &  0   &  1 \\
                    1   &  1   &  1 \\
                    1   &  1   &  1\\
\end{pmatrix}.
\]  
Then,

\[u^{n} =  \begin{pmatrix} \vspace{0.2cm}
                   E_{n-1}  & E_{n-1} & E_{n-1} + E_{n-2}\\ \vspace{0.2cm}
                   E_{n}  & E_{n}     & E_{n} + E_{n-1} \\
                   E_{n}  & E_{n}     & E_{n} + E_{n-1} \\
\end{pmatrix}  \]   
where \( n\in \mathbb{Z}^{+} \) and  \( E_{n} \) denotes the Pell numbers, defined by the recurrence relation  
\[
E_n = 2E_{n-1} + E_{n-2},
\]
with initial conditions \( E_0 = 0 \) and \( E_1 = 1 \). 
\end{lem}

\begin{proof}
We will use the principle of mathematical induction (PMI). When $n = 2$,
\[
u^{2} = \begin{bmatrix}  
E_{1}  & E_{1} & E_{1} + E_{0}\\ 
E_{2}  & E_{2} & E_{2} + E_{1} \\  
E_{2}  & E_{2} & E_{2} + E_{1}\\ 
\end{bmatrix} 
= 
\begin{bmatrix} 
1 & 1 & 1\\ 
2 & 2 & 3\\  
2 & 2 & 3\\ 
\end{bmatrix}
\]
so the result is true. We assume it is true for any positive integer $n = k$:

\[u^{k} =  \begin{pmatrix} \vspace{0.2cm}
                   E_{k-1}& E_{k-1}   & E_{k-1} + E_{k-2}\\ \vspace{0.2cm}
                   E_{k}  & E_{k}     & E_{k} + E_{k-1} \\
                   E_{k}  & E_{k}     & E_{k} + E_{k-1} \\
\end{pmatrix}\]  
Now, we show that it is true for $n = k + 1$. Then, we can write
\[u^{k+1} = u^k u 
= \begin{pmatrix} \vspace{0.2cm}
                   E_{k-1}& E_{k-1}   & E_{k-1} + E_{k-2}\\ \vspace{0.2cm}
                   E_{k}  & E_{k}     & E_{k} + E_{k-1} \\
                   E_{k}  & E_{k}     & E_{k} + E_{k-1} \\
\end{pmatrix}
\begin{pmatrix} 
                    0   &  0   &  1 \\
                    1   &  1   &  1 \\
                    1   &  1   &  1\\
\end{pmatrix} 
\]
\[ 
= 
\begin{pmatrix} \vspace{0.2cm}
                   E_{k}   & E_{k}   & E_{k} + E_{k-1}\\ \vspace{0.2cm}
                   E_{k+1} & E_{k+1} & E_{k+1} + E_{k} \\
                   E_{k+1} & E_{k+1} & E_{k+1} + E_{k} \\
\end{pmatrix}
\]
and the result follows due to the recurrence relation $E_k = 2E_{k-1} + E_{k-2}$
\end{proof}

Analogously, the following lemma can be proven by mathematical induction.

\begin{lem}\label{l:2} 
Let $u$ be the matrix   
\[
u = \begin{pmatrix} 
                    0   &  1   &  1 \\
                    0   &  1   &  1 \\
                    1   &  1   &  1\\
\end{pmatrix}.
\]  
Then, 

\[u^{n} =  \begin{pmatrix} \vspace{0.2cm}
                   E_{n-1} & E_{n} & E_{n}  \\  \vspace{0.2cm}
                   E_{n-1} & E_{n} & E_{n} \\
                   E_{n-1}+ E_{n-2}& E_{n}+ E_{n-1} & E_{n}+ E_{n-1}  \\
\end{pmatrix}  \]

where \( n\in \mathbb{Z}^{+} \) and  \( E_{n} \) denotes the Pell numbers, defined by the recurrence relation  
\[
E_n = 2E_{n-1} + E_{n-2},
\]
with initial conditions \( E_0 = 0 \) and \( E_1 = 1 \). 
\end{lem}

\begin{cor}

The following matrices are similar, each being a Pell matrix with determinant zero.

\[
\begin{pmatrix} 
                    0   &  0   &  1 \\
                    1   &  1   &  1 \\
                    1   &  1   &  1\\
\end{pmatrix},
\begin{pmatrix} 
                    0   &  1   &  1 \\
                    0   &  1   &  1 \\
                    1   &  1   &  1\\
\end{pmatrix}.
\] 
Then, for every \( n \in \mathbb{N} \), we have that
\[
\operatorname{Tr}(u^n) = 2\left( E_{n-1} + E_n\right) 
\]
\end{cor}

A formula for $E_{m+n}$ is already known in the literature(see, e.g., \cite{Koshy2014}, p. 158). The identity we present here is algebraically equivalent, but provides an alternative form which, to the best of our knowledge, has not been emphasized before.

\begin{cor}
For $n\geq 1$ and  $m\geq 1$, we have the following identity:
\begin{align}
E_{m+n-1}                &= 2 E_{m-1}E_{n-1} +   E_{n}( E_{m-1}+E_{m-2} ),\\
E_{m+n-1}+E_{m+n-2}      &=2 E_{m-1}E_{n}+  ( E_{m-1}+E_{m-2} )( E_{n}+E_{n-1} ),\\
E_{m+n}                  &=E_{m}(E_{n}+E_{n-1})+ E_{n}(E_{m}+E_{m-1}),\\
E_{m+n}+E_{m+n-1}        &=2 E_{m}E_{n} +   (E_{m}+E_{m-1})(E_{n}+E_{n-1}). 
\end{align}
\end{cor}

\begin{proof}

For \( m, n \geq 1 \), we know that \( U^{m+n} = U^m U^n \). Since we have defined \( U^n \) as follows, the same expression in matrix form is:

\[
U^m = \begin{pmatrix} \vspace{0.2cm}
                   E_{m-1}  & E_{m-1} & E_{m-1} + E_{m-2}\\ \vspace{0.2cm}
                   E_{m}  & E_{m}     & E_{m} + E_{m-1} \\
                   E_{m}  & E_{m}     & E_{m} + E_{m-1} \\
\end{pmatrix}, 
\] 

\[
U^{n} =  \begin{pmatrix} \vspace{0.2cm}
                   E_{n-1}  & E_{n-1} & E_{n-1} + E_{n-2}\\ \vspace{0.2cm}
                   E_{n}    & E_{n}     & E_{n} + E_{n-1} \\
                   E_{n}    & E_{n}     & E_{n} + E_{n-1} \\
\end{pmatrix} . \]   
The product \( U^m U^n \) is represented by the following \( 3 \times 3 \) matrix:
\[
U^m U^n =
\begin{bmatrix}
A & A & B \\
C & C & D \\
C & C & D
\end{bmatrix}
\]
where:

\[
\begin{aligned}
A &= 2 E_{m-1}E_{n-1} +   E_{n}( E_{m-1}+E_{m-2} ),   \\
B &=  2 E_{m-1}E_{n}+  ( E_{m-1}+E_{m-2} )( E_{n}+E_{n-1} ),  \\
C &=  2 E_{m}E_{n} +   E_{m}E_{n-1}+E_{n}E_{m-1},\\  
D &=  2 E_{m}E_{n} +   (E_{m}+E_{m-1})(E_{n}+E_{n-1}).\\  
\end{aligned}
\] 
On the other hand, we have:

\[
U^{m+n} = 
 \begin{pmatrix} \vspace{0.2cm}
                   E_{m+n-1}  & E_{m+n-1}  & E_{m+n-1} + E_{m+n-2}\\ \vspace{0.2cm}
                   E_{m+n}    & E_{m+n}     & E_{m+n} + E_{m+n-1} \\
                   E_{m+n}    & E_{m+n}     & E_{m+n} + E_{m+n-1} \\
\end{pmatrix}, 
\]
Equating the two matrices obtained via matrix multiplication yields the identities stated in the corollary.
\end{proof}

\subsection{Diagonalization of the Generating Matrix and Binet’s Formula}

The Binet-type formula~\ref{i:1} for the Pell numbers and the generalized Pell sequence~\ref{i:2} are classical results and follow from the general theory of Lucas sequences; see, for example, \cite{Lucas1878,Horadam1994}, as well as modern treatments in \cite{Koshy2014,Bicknell1975}. We include them here for completeness. Furthermore, identity~\ref{i:4} can be found in \cite{Horadam1994}.

\begin{thm}
Let $n$ be an integer. The well-known Binet-like formulas of the Pell numbers and the generalized Pell sequence are
\begin{align}
\label{i:1} E_n &= \dfrac{1}{2\sqrt{2}}\left[ (1+\sqrt{2})^{n} - (1-\sqrt{2})^{n}  \right],  \\ 
\label{i:2} Q_{n-1} &= \dfrac{1}{2}\left[ (1+\sqrt{2})^{n} + (1-\sqrt{2})^{n}\right], \\  
\label{i:4}   Q_{n-1}  &= E_{n}+E_{n-1},  \\  
\label{i:3}  E_{n-1}+E_{n-2} &= 2E_{n}-Q_{n-1}. 
\end{align}

\end{thm}

\begin{proof}
Let the matrix $U$ be as in Lemma~\ref{l:1}. If we calculate the eigenvalues and eigenvectors of the matrix $U$ are
\[
\lambda_1=0, \quad \lambda_2=1-\sqrt{2}, \quad \lambda_3=1+\sqrt{2}
\]
and

\[v_1=
\begin{pmatrix}
-1 \\
1 \\
0
\end{pmatrix}, \, v_2=
\begin{pmatrix}
-1-\sqrt{2} \\
1 \\
1
\end{pmatrix}, \, v_3=
\begin{pmatrix}
-1+\sqrt{2} \\
1 \\
1
\end{pmatrix}
\]
respectively. Then, we can diagonalize of the matrix $U$ by
\[
D = P^{-1}UP
\]
where
\[
P = (v_1, v_2, , v_3) = 
\begin{pmatrix}
-1 & -\lambda_3 & -\lambda_2 \\
1 & 1 & 1 \\
0 & 1 & 1
\end{pmatrix},
\, 
P^{-1} = 
\begin{pmatrix}\vspace{0.2cm}
0 & 1 & -1 \\ \vspace{0.2cm}
-\dfrac{1}{2\sqrt{2}} & -\dfrac{1}{2\sqrt{2}} & \dfrac{1}{2} \\
\dfrac{1}{2\sqrt{2}}  & \dfrac{1}{2\sqrt{2}} & \dfrac{1}{2}
\end{pmatrix}
\]
and
\[
D = \text{diag}(\lambda_1,\lambda_2,,\lambda_3) =
\begin{pmatrix}
0 & 0 & 0 \\
0 & 1-\sqrt{2} & 0 \\
0 & 0 & 1+\sqrt{2}
\end{pmatrix}.
\]
From the properties of the similar matrices, we can write 
\begin{equation}\label{i:35}
U^n = P D^n P^{-1}
\end{equation}
where $n$ is any integer and
\[
D^n = 
\begin{pmatrix}
0 & 0 & 0 \\  \vspace{0.2cm}
0 & (1-\sqrt{2})^{n} & 0 \\
0 & 0 & (1+\sqrt{2})^{n}
\end{pmatrix} = \begin{pmatrix}
0 & 0 & 0 \\
0 & Q_{n-1}-E_{n}\sqrt{2} & 0 \\
0 & 0 & Q_{n-1} + E_{n}\sqrt{2} 
\end{pmatrix}.
\]
By equation~\ref{i:35}, we get
\[
U^{n}=\begin{pmatrix} \vspace{0.2cm}
                   E_{n-1}  & E_{n-1} & E_{n-1} + E_{n-2}\\ \vspace{0.2cm}
                   E_{n}  & E_{n}     & E_{n} + E_{n-1} \\
                   E_{n}  & E_{n}     & E_{n} + E_{n-1} \\
\end{pmatrix}
=
\begin{bmatrix}\vspace{0.2cm}
Q_{n-1}-E_{n}  & Q_{n-1}-E_{n} &  -Q_{n-1}+2E_{n}\\ \vspace{0.2cm}
         E_{n} & E_{n}         &  Q_{n-1}    \\ \vspace{0.2cm}
         E_{n} &  E_{n}        &  Q_{n-1}   
\end{bmatrix}
\]
\[
= 
\begin{pmatrix}
\dfrac{1}{2\sqrt{2}}\lambda_2\lambda_3\left(\lambda_2^{n-1}-\lambda_3^{n-1} \right)   & \dfrac{1}{2\sqrt{2}}\lambda_2\lambda_3\left(\lambda_2^{n-1}-\lambda_3^{n-1} \right)  & \dfrac{1}{2}\lambda_2\lambda_3\left(-\lambda_3^{n-1}-\lambda_2^{n-1} \right) \\  \vspace{0.2cm}
\dfrac{1}{2\sqrt{2}}\left( \lambda_3^{n} - \lambda_2^{n}  \right) & \dfrac{1}{2\sqrt{2}}\left( \lambda_3^{n} - \lambda_2^{n}  \right)  & \dfrac{1}{2}\left( \lambda_3^{n} + \lambda_2^{n}\right) \\
\dfrac{1}{2\sqrt{2}}\left( \lambda_3^{n} - \lambda_2^{n}  \right)  & \dfrac{1}{2\sqrt{2}}\left( \lambda_3^{n} - \lambda_2^{n}  \right) & \dfrac{1}{2}\left( \lambda_3^{n} + \lambda_2^{n}  \right)
\end{pmatrix}.
\]
Thus, the proof is completed. 
\end{proof}

\begin{thm}
Let $n$ be an integer. The well-known Binet-like formula of the Pell numbers is
\begin{align}
E_{n-1} + E_{n-2} &= \dfrac{1}{2}\left[(1-\sqrt{2})^{n-1}+(1+\sqrt{2})^{n-1} \right], \\ 
E_{n} + E_{n-1} &= \dfrac{1}{2}\left[(1-\sqrt{2})^{n}+(1+\sqrt{2})^{n} \right], \\ 
E_{n-1}  &= \dfrac{1}{2\sqrt{2}}\left[(1+\sqrt{2})^{n-1}-(1-\sqrt{2})^{n-1} \right]. 
\end{align}

\end{thm}

\begin{proof}
Let the matrix $U$ be as in Lemma~\ref{l:2}. If we calculate the eigenvalues and eigenvectors of the matrix $U$ are
\[
\lambda_1=0, \quad \lambda_2=1-\sqrt{2}, \quad \lambda_3=1+\sqrt{2}
\]
and

\[v_1=
\begin{pmatrix}
0 \\ \\
-1 \\ \\
1
\end{pmatrix}, \, v_2=
\begin{pmatrix}
-\dfrac{\sqrt{2}}{2} \\ \\
-\dfrac{\sqrt{2}}{2}  \\ \\
1
\end{pmatrix}, \, v_3=
\begin{pmatrix}
\dfrac{\sqrt{2}}{2} \\ \\
\dfrac{\sqrt{2}}{2}\\ \\
1
\end{pmatrix}
\]
respectively. Then, we can diagonalize of the matrix $U$ by
\[
D = P^{-1}UP
\]
where
\[
P = (v_1, v_2, , v_3) = 
\begin{pmatrix}
0 & \dfrac{\lambda_2  -1 }{2}&  \dfrac{\lambda_3  -1 }{2} \\ \\
-1 & \dfrac{\lambda_2  -1 }{2} & \dfrac{\lambda_3  -1 }{2} \\ \\
1 & 1 & 1
\end{pmatrix},
\quad 
P^{-1} = 
\begin{pmatrix}\vspace{0.2cm}
1 & -1 & 0 \\ \vspace{0.2cm}
\dfrac{-\sqrt{2}-1}{2} & \dfrac{1}{2} & \dfrac{1}{2} \\
\dfrac{\sqrt{2}-1}{2} & \dfrac{1}{2} & \dfrac{1}{2}
\end{pmatrix}
\]
and
\[
D = \text{diag}(\lambda_1,\lambda_2,,\lambda_3) =
\begin{pmatrix}
0 & 0 & 0 \\
0 & 1-\sqrt{2} & 0 \\
0 & 0 & 1+\sqrt{2}
\end{pmatrix}.
\]
From the properties of the similar matrices, we can write 
\begin{equation}\label{i:36}
U^n = P D^n P^{-1}
\end{equation}
where $n$ is any integer and
\[
D^n = 
\begin{pmatrix}
0 & 0 & 0 \\  \vspace{0.2cm}
0 & (1-\sqrt{2})^{n} & 0 \\
0 & 0 & (1+\sqrt{2})^{n}
\end{pmatrix}=\begin{pmatrix}
0 & 0 & 0 \\
0 & Q_{n-1}-E_{n}\sqrt{2} & 0 \\
0 & 0 & Q_{n-1} + E_{n}\sqrt{2} 
\end{pmatrix}.
\]
By equation~\ref{i:36}, we get
\[
U^{n}=\begin{pmatrix} \vspace{0.2cm}
                   E_{n-1}  & E_{n}   &  E_{n}\\ \vspace{0.2cm}
                   E_{n-1}  & E_{n}   &  E_{n} \\
E_{n-1}+E_{n-2}  & E_{n} + E_{n-1}    & E_{n} + E_{n-1} \\
\end{pmatrix}
=
\begin{bmatrix}\vspace{0.2cm}
Q_{n-1}-E_{n}  & E_{n}  &  E_{n}\\ \vspace{0.2cm}
Q_{n-1}-E_{n}  & E_{n}  &  E_{n}    \\ \vspace{0.2cm}
2E_{n}-Q_{n-1} & Q_{n-1}&  Q_{n-1}   
\end{bmatrix}
\]
\[
= 
\begin{pmatrix}
\dfrac{1}{2\sqrt{2}}\lambda_2\lambda_3\left(\lambda_2^{n-1}-\lambda_3^{n-1} \right)  & \dfrac{1}{2\sqrt{2}}\left( \lambda_3^{n} - \lambda_2^{n}  \right)  & \dfrac{1}{2\sqrt{2}}\left( \lambda_3^{n} - \lambda_2^{n}  \right) \\ \\ 
\dfrac{1}{2\sqrt{2}}\lambda_2\lambda_3\left(\lambda_2^{n-1}-\lambda_3^{n-1} \right) & \dfrac{1}{2\sqrt{2}}\left( \lambda_3^{n} - \lambda_2^{n}  \right)  & \dfrac{1}{2\sqrt{2}}\left( \lambda_3^{n} - \lambda_2^{n}  \right) \\ \\
-\dfrac{1}{2}\lambda_2\lambda_3\left(\lambda_2^{n-1}+\lambda_3^{n-1} \right) & \dfrac{1}{2}\left( \lambda_3^{n} + \lambda_2^{n}  \right) & \dfrac{1}{2}\left( \lambda_3^{n} + \lambda_2^{n}  \right)
\end{pmatrix}.
\]
Thus, the proof is completed. 
\end{proof}

\section{Recurrences Generated by Binary $3 \times 3$ Matrices with Determinant One}
Consider the generating matrix
\[
u = \begin{pmatrix}0 & 1 & 1\\ 1 & 0 & 1\\ 1 & 1 & 1\end{pmatrix}.
\]
In this section, we present the Pell \textit{U}-matrix with determinant~1, a matrix representation of Pell numbers. We use it to compute powers \(U^n\), determinants, inverses and Cassini-like identities; to derive the characteristic roots and the Binet formula for the sequence \(b_n\); to establish identities involving the Pell sequence together with the generalized Pell sequence,  and \(r_n\); and to obtain summation formulas for Pell numbers, \(Q_n\), and \(b_n\).

\subsection{The Matrix Representation of Recurrences and Their Identities}

\begin{lem}\label{l:4} 
Let \( E_n \) be the Pell numbers defined by the recurrence
\[
E_n = 2E_{n-1} + E_{n-2}, \quad E_0 = 0,\ E_1 = 1,
\]
and let \( b_n \) be the sequence defined by
\[
b_n = b_{n-1} + 3b_{n-2} + b_{n-3}, \quad b_0 = 0,\ b_1 = 1,\ b_2 = 1.
\]
Then, for all \( n \geq 1 \),
\[
E_n = b_n + b_{n-1}.
\]
\end{lem}

\begin{proof}
We will use the principle of mathematical induction (PMI). When $n = 1$,
\[
E_1 = b_1 + b_{0}=1+0=1.
\] 
so the result is true. We assume it is true for any positive integer $n=k$:
\[
E_k = b_k + b_{k-1}.
\] 
Now, we show that it is true for $n = k + 1$. Using the recurrence relation for \( E_n \), we compute:
\begin{align*}
E_{k+1} &= 2E_k + E_{k-1} \\
        &= 2(b_k + b_{k-1}) + (b_{k-1} + b_{k-2}) \quad \text{(by the induction hypothesis)} \\
        &= 2b_k + 2b_{k-1} + b_{k-1} + b_{k-2} \\
        &= b_k + (b_k + 3b_{k-1} + b_{k-2}) \\
        &= b_k + b_{k+1} \quad \text{(by the recurrence for $b_{n}$)} \\
        &= b_{k+1}  + b_{k}               
\end{align*}
Thus, the identity holds for \( n = k+1 \), and by induction, it holds for all positive integers \( n \).
\end{proof}

\begin{lem}\label{l:5} 
Let $u$ be the symmetric matrix
\[
u = \begin{pmatrix} 
                    0   &  1   &  1 \\
                    1   &  0   &  1 \\
                    1   &  1   &  1\\
\end{pmatrix}.
\]  
Then, 

\[u^{n} =  \begin{pmatrix} \vspace{0.2cm}
                   E_{n}+E_{n-1}- b_{n} & b_{n} & E_{n} \\ \vspace{0.2cm}
                   b_{n} & E_{n}+E_{n-1}-b_{n} & E_{n} \\
                   E_{n} & E_{n} & E_{n}+E_{n-1} \\
\end{pmatrix}  \]   
where \( n\in \mathbb{Z}^{+} \) and  \( E_{n} \) denotes the Pell numbers, defined by the recurrence relation  
\[
E_n = 2E_{n-1} + E_{n-2},
\]
with initial conditions \( E_0 = 0 \) and \( E_1 = 1 \). Similarly, the numbers $b_{n}$ satisfy 

\[b_{n} = b_{n-1} + 3b_{n-2} + b_{n-3},\] with  \(b_{0} =0\), \( b_{1} =1\), \( b_{2} =1 .\) 

\end{lem}

\begin{proof}

by the principle of mathematical induction, When $n = 1$, 

\[
u^1 = \begin{pmatrix} 
                    0   &  1   &  1 \\
                    1   &  0   &  1 \\
                    1   &  1   &  1\\
\end{pmatrix},
\] 
so the result is true. Now, assume it is true for an arbitrary positive integer $n$: 

\[
u^n u = 
\begin{pmatrix} \vspace{0.2cm}
    E_n + E_{n-1} - b_n & b_n & E_n \\ \vspace{0.2cm}
    b_n & E_n + E_{n-1} - b_n & E_n \\
    E_n & E_n & E_n + E_{n-1}
\end{pmatrix}\]
Then we compute the product \( u^n u \) using the given matrices:

\[
u^n u = 
\begin{pmatrix} \vspace{0.2cm}
    E_n + E_{n-1} - b_n & b_n & E_n \\ \vspace{0.2cm}
    b_n & E_n + E_{n-1} - b_n & E_n \\
    E_n & E_n & E_n + E_{n-1}
\end{pmatrix}
\begin{pmatrix}
    0 & 1 & 1 \\
    1 & 0 & 1 \\
    1 & 1 & 1
\end{pmatrix}
\]
Carrying out the matrix multiplication, we obtain:

\[
u^n u = 
\begin{pmatrix} \vspace{0.2cm}
    b_n + E_n & 2E_n + E_{n-1} - b_n & 2E_n + E_{n-1} \\ \vspace{0.2cm}
    2E_n + E_{n-1} - b_n & b_n + E_n & 2E_n + E_{n-1} \\
    2E_n + E_{n-1} & 2E_n + E_{n-1} & 3E_n + E_{n-1}
\end{pmatrix}
\]
Recall that the sequence \( (E_n) \) satisfies the recurrence relation
\[
E_{n+1} = 2E_n + E_{n-1},
\]
from which it follows that
\[
E_{n+1} + E_n = 3E_n + E_{n-1}.
\]
Substituting into the matrix, we get:
\[
u^n u = 
\begin{pmatrix} \vspace{0.2cm}
    b_n + E_n & E_{n+1} - b_n & E_{n+1} \\ \vspace{0.2cm}
    E_{n+1} - b_n & b_n + E_n & E_{n+1} \\
    E_{n+1} & E_{n+1} & E_{n+1} + E_n
\end{pmatrix}
\]
Now, we use the identity given in Lemma~\ref{l:4}.
\[
E_{n+1} = b_{n+1} + b_n,
\]
we can write
\[
E_{n+1} + E_n - b_{n+1} = b_n + E_n.
\]
Hence, the matrix becomes:

\[
u^n u = 
\begin{pmatrix} \vspace{0.2cm}
    E_{n+1} + E_n - b_{n+1} & b_{n+1} & E_{n+1} \\ \vspace{0.2cm}
    b_{n+1} & E_{n+1} + E_n - b_{n+1} & E_{n+1} \\
    E_{n+1} & E_{n+1} & E_{n+1} + E_n
\end{pmatrix} = u^{n+1}
\]
This completes the proof.
\end{proof}

\begin{lem}\label{l:6} 
Let \( E_n \) be the Pell numbers defined by the recurrence
\[
E_n = 2E_{n-1} + E_{n-2}, \quad E_0 = 0,\ E_1 = 1,
\]
Then, for all \( n \geq 1 \),
\[
3E_{n-1}+ E_n = E_{n+1} - E_{n-2}.
\]
\end{lem}
\begin{proof}
\begin{align*}
    3E_{n-1} + E_n 
    &= 2E_{n-1} + (E_n + E_{n-1}) \\
    &= (E_n - E_{n-2}) + (E_n + E_{n-1})  \quad \text{(by the recurrence for $E_{n}$)} \\
    &= (2E_n + E_{n-1}) - E_{n-2} \\
    &= E_{n+1} - E_{n-2}   \quad \text{(by the recurrence for $E_{n}$)}.
\end{align*}
This concludes the proof.
\end{proof}

\begin{rem}
We recall the following identity, sometimes referred to as Simpson’s formula for Pell numbers (see Horadam \cite{Horadam1961}, Identity 30) :

\begin{equation}\label{eq:simpson}
E_n E_{n-2} - E_{n-1}^2 = (-1)^{n-1}.
\end{equation}

This identity will be used in the final step of the proof below. 
Although this result was first proved in \cite{KilicTasci2005}, 
we present here an alternative proof, which highlights the 
connection with Simpson's formula. 
\end{rem}

\begin{lem}[see E. Kilic and D. Tasci \cite{KilicTasci2005}]\label{l:7} 
Let \( E_n \) be the Pell numbers defined by the recurrence
\[
E_n = 2E_{n-1} + E_{n-2}, \quad E_0 = 0,\ E_1 = 1,
\]
Then, for all \( n \geq 1 \),
\[
E_{n}^{2}- E_{n-1}^{2} - 2E_{n}E_{n-1}=(-1)^{n-1}.
\]
\end{lem}

\begin{proof}
\begin{align*}
E_{n}^{2} - E_{n-1}^{2} - 2E_{n}E_{n-1}
    &= (E_{n}^{2} - E_{n}E_{n-1}) - (E_{n-1}^{2} + E_{n}E_{n-1}) \\
    &= E_{n}(E_{n} - E_{n-1}) - E_{n-1}(E_{n-1} + E_{n}) \\
    &= E_{n}(E_{n-1} + E_{n-2}) - E_{n-1}(E_{n-1} + E_{n})\\ 
    &\quad \text{(by the recurrence for $E_n$)} \\
    &= E_{n}E_{n-2} - E_{n-1}^{2} \\
    &\quad \text{(by Simpson’s identity for Pell numbers, see \eqref{eq:simpson})}\\
    &= (-1)^{n-1}.
\end{align*}
This concludes the proof.
\end{proof}

\begin{rem}
Although identities of the form $Q_n = \alpha P_n + \beta P_{n+1}$ 
for sequences $Q_n$ satisfying the Pell recurrence are classical  (see, \cite{TrojnarSpelina2019}), the specific choice of initial values $Q_0=1, Q_1=3$ 
yields the identity
\[
Q_n = P_n + P_{n+1}.
\] 
For completeness, we present the short derivation below.
\end{rem}

\begin{lem}\label{l:13} 
Let \( E_n \) be the Pell numbers defined by the recurrence
\[
E_n = 2E_{n-1} + E_{n-2}, \quad E_0 = 0,\ E_1 = 1,
\]
and let \( Q_n \) be the generalized Pell sequence defined by
\[
 Q_n = 2 \, Q_{n-1} + Q_{n-2} \, \quad Q_0 = 1 ,\,  Q_1 = 3 .
\]

Then, for all \( n \geq 1 \),
\[
 Q_n = E_n + E_{n+1}.
\]
\end{lem}

\begin{proof}
We will use the principle of mathematical induction (PMI). When $n = 1$,
\[
E_1 = E_1 + E_{2}=1+2=3.
\] 
so the result is true. We assume it is true for any positive integer $n=k$:
\[
Q_k = E_k + E_{k+1}.
\] 
Now, we show that it is true for $n = k + 1$. Using the recurrence relation for \( E_n \), we compute:
\begin{align*}
Q_{k+1} &= 2Q_k + Q_{k-1} \\
        &= 2(E_k + E_{k+1}) + (E_{k-1} + E_{k}) \quad \text{(by the induction hypothesis)} \\
        &= (2E_k +  E_{k-1}) + ( 2E_{k+1}+E_{k} ) \\
        &= E_{k+1} +  E_{k+2} \quad \text{(by the recurrence for $E_{n}$)} \\              
\end{align*}
Thus, the identity holds for \( n = k+1 \), and by induction, it holds for all positive integers \( n \).
\end{proof}

\begin{cor}
Let \( u \) be the Pell matrix defined by
\[
u = \begin{pmatrix} 
                    0   &  1   &  1 \\
                    1   &  0   &  1 \\
                    1   &  1   &  1\\
\end{pmatrix}.
\] Then, for every \( n \in \mathbb{N} \), we have that
\[
\det(u^n) =(Q_{n-1}-2b_{n})\,(-1)^{n} = 1.
\]
\end{cor}

\begin{proof}
It is easy to see that
\[
\det(u) = 1.
\]
Then, it can be written
\begin{align*}
\det(u^n) &= \det(F) \cdot \det(F) \cdot \dots \cdot \det(F) \\
          &= (1)^n=1.
\end{align*}
If \( x = E_{n} \), \( y = E_{n-1} \), and \( z = b_{n}\), then the determinant of the matrix \( u^n \) given in Lemma~\ref{l:5} is

\begin{align*}
\left|
\begin{array}{ccc}
x + y - z & z & x \\
z & x + y - z & x \\
x & x & x + y
\end{array}
\right|
&= -x^3 + y^3 + 3xy^2 + x^2y + 2x^2z - 2y^2z - 4xyz\\
&= -(x+y-2z)\,(x^2 - 2xy - y^2)\\
&= -(E_{n}+E_{n-1}-2b_{n})\,(E_{n}^2 - 2E_{n}E_{n-1} - E_{n-1}^2) \\
&\quad \text{(by the identities for $Q_{n-1}$; see Lemmas~\ref{l:13},~\ref{l:7})} \\
&= -(Q_{n-1}-2b_{n})\,(-1)^{n-1}  \\
&= (Q_{n-1}-2b_{n})\,(-1)^n.
\end{align*}
Thus, \[(Q_{n-1}-2b_{n})(-1)^{n}=1 .\] for all $n\geq 1.$ 

\end{proof}

\begin{cor}
For $n\geq 1$ and  $m\geq 1$, we have the following identity:
\begin{align}
E_{m+n}+E_{m+n-1}- b_{m+n}&=(E_m + E_{m-1} - b_m)(E_n + E_{n-1} - b_n) + b_m b_n + E_{m} E_{n} ,\label{id:7}\\
b_{m+n} &= b_n(E_m + E_{m-1}-b_m) + b_m(E_n + E_{n-1} - b_n) + E_{m} E_{n},\label{id:8} \\
E_{m+n} &= E_n(E_m + E_{m-1}) + E_m(E_n + E_{n-1}),\label{id:9}\\
E_{m+n}+E_{m+n-1}  &= (E_m + E_{m-1})(E_n + E_{n-1}) + 2E_{m} E_n .\label{id:10}
\end{align}
\end{cor}

\begin{proof}

For \( m, n \geq 1 \), we know that \( U^{m+n} = U^m U^n \). Since we have defined \( U^n \) as follows, the same expression in matrix form is:

\[
U^m = \begin{pmatrix} \vspace{0.2cm}
                   E_{m}+E_{m-1}- b_{m} & b_{m} & E_{m} \\ \vspace{0.2cm}
                   b_{m} & E_{m}+E_{m-1}-b_{m} & E_{m} \\
                   E_{m} & E_{m} & E_{m}+E_{m-1} \\
\end{pmatrix}, 
\] 

\[
U^{n} =  \begin{pmatrix} \vspace{0.2cm}
                   E_{n}+E_{n-1}- b_{n} & b_{n} & E_{n} \\ \vspace{0.2cm}
                   b_{n} & E_{n}+E_{n-1}-b_{n} & E_{n} \\
                   E_{n} & E_{n} & E_{n}+E_{n-1} \\
\end{pmatrix}  \]   
The product \( U^m U^n \) is given by the following symmetric \( 3 \times 3 \) matrix:
\[
U^m U^n =
\begin{bmatrix}
A & B & C \\
B & A & C \\
C & C & D
\end{bmatrix}
\]
where:

\[
\begin{aligned}
A &= (E_m + E_{m-1} - b_m)(E_n + E_{n-1} - b_n) + b_m b_n + E_{m} E_{n} \\
B &= b_n(E_m + E_{m-1}-b_m) + b_m(E_n + E_{n-1} - b_n) + E_{m} E_{n}, \\
C &= E_n(E_m + E_{m-1}) + E_m(E_n + E_{n-1}), \\
D &= (E_m + E_{m-1})(E_n + E_{n-1}) + 2E_{m} E_n.
\end{aligned}
\] 
On the other hand, we have:

\[
U^{m+n} = 
 \begin{pmatrix} \vspace{0.2cm}
                   E_{m+n}+E_{m+n-1}- b_{m+n} & b_{m+n} & E_{m+n} \\ \vspace{0.2cm}
                   b_{m+n} & E_{m+n}+E_{m+n-1}-b_{m+n} & E_{m+n} \\
                   E_{m+n} & E_{m+n} & E_{m+n}+E_{m+n-1} \\
\end{pmatrix}, 
\]
Equating the two matrices obtained via matrix multiplication yields the identities stated in the corollary.
\end{proof}

\begin{cor}

For $n\geq 1$ and  $m\geq 1$, we have the following identity:

\begin{align}
Q_{m+n-1}&= Q_{n-1}Q_{m-1}+2\,E_{n}E_{m},\label{id:11} \\
E_{m+n}  &= E_{m}Q_{n-1}+\,E_{n}Q_{m-1},\label{id:12} \\
b_{m+n}  &= b_{n}Q_{m-1}+b_{m}Q_{n-1}-2\,b_{n}b_{m}+\,E_{n}E_{m}\label{id:13} .
\end{align}

\end{cor}

\begin{proof} 
Let \( Q_n \) denote the generalized Pell sequence, defined by
\[
Q_n = 2Q_{n-1} + Q_{n-2}, \quad Q_0 = 1,\ Q_1 = 3,
\]
and let \( E_n \) denote the Pell numbers, defined by
\[
E_n = 2E_{n-1} + E_{n-2}, \quad E_0 = 0,\ E_1 = 1,\ E_2 = 2.
\]
From Lemma~\ref{l:13}, we have
\[
E_{m} + E_{m-1} = Q_{m-1}.
\]
Substituting this relation into identity~\ref{id:10},
\[
E_{m+n}+E_{m+n-1}  = (E_m + E_{m-1})(E_n + E_{n-1}) + 2E_{m} E_n,
\]
we obtain
\[
\begin{aligned}
Q_{m+n-1} 
&= Q_{n-1} Q_{m-1} + 2E_{n} E_{m}, \qquad n \geq 1,\ m \geq 1.
\end{aligned}
\]
Thus, formula~\ref{id:11} is verified.

\medskip

Similarly, substituting the identity for \(Q_{n}\) (see Lemma~\ref{l:13}) into identity~\ref{id:9},
\[
E_{m+n} = E_n(E_m + E_{m-1}) + E_m(E_n + E_{n-1}),
\]
we find
\[
\begin{aligned}
E_{m+n} &= E_{n}Q_{m-1} +  E_{m}Q_{n-1}, \qquad n \geq 1,\ m \geq 1.
\end{aligned}
\]
Hence, formula~\ref{id:12} follows.

\medskip 
Finally, substituting into identity~\ref{id:8},
\[
b_{m+n} = b_n(E_m + E_{m-1}-b_m) + b_m(E_n + E_{n-1} - b_n) + E_{m} E_{n},
\]
we obtain
\[
\begin{aligned}
b_{m+n} &= b_n(E_m + E_{m-1}) + b_m(E_n + E_{n-1})  - 2\,b_n b_m + E_{m} E_{n}\\
        &= b_n Q_{m-1} + b_m Q_{n-1} - 2b_n b_m + E_{m} E_{n}.
\end{aligned}
\]
Therefore, formula~\ref{id:13} is established:
\[
b_{m+n} = b_n Q_{m-1} + b_m Q_{n-1} - 2b_n b_m + E_{m} E_{n}, 
\qquad n \geq 1,\ m \geq 1.
\]
\end{proof}

\begin{lem}\label{l:11} 
Let \( r_n \) be the numbers defined by the recurrence
\[
r_n = 2 \, r_{n-1} + r_{n-2} + 1, \quad r_1 = 0,\ r_2 = 1,
\]
and let \( b_n \) be the sequence defined by
\[
b_n = b_{n-1} + 3 \, b_{n-2} + b_{n-3}, \quad b_0 = 0,\ b_1 = 1,\ b_2 = 1.
\]
Then, for all \( n \geq 1 \),

\[
b_n = 
\begin{cases}
r_{n}       & \text{if } n \text{ is even}, \\
r_{n}+1   & \text{if } n \text{ is odd}.
\end{cases}
\]
\end{lem}

\begin{proof}
We proceed by induction. For $n = 1$ and $n = 2$, we have
\[
b_1 = r_1 + 1 = 1, \quad b_2 = r_2 = 1,
\]
so the statement holds in these cases.
Now, assume that the formula holds for all integers up to $n$; that is,
\[
b_n = 
\begin{cases}
r_{n}       & \text{if } n \text{ is even}, \\
r_{n}+1   & \text{if } n \text{ is odd}.
\end{cases}
\]
We want to show that it holds for \( n+1 \).

\textbf{Case 1:} Suppose \( n \) is odd. Then \( n+1 \) is even, and we compute:
\begin{align*}
b_{n+1} &= b_n + 3\,b_{n-1} + b_{n-2} \\
       &= (r_{n} + 1) + 3\,r_{n-1} + (r_{n-2} + 1) \quad \text{(by induction hypothesis)} \\
       &= r_{n} + r_{n-1} + 1 + ( 2\,r_{n-1} + r_{n-2} + 1 ) \\
       &= r_{n} + r_{n-1} + 1 + r_{n} \quad \text{(by the recurrence for $r_n$ )} \\
       &= 2\,r_{n} + r_{n-1} + 1 \\
       &= r_{n+1}.
\end{align*}

\textbf{Case 2:} Suppose \( n \) is even. Then \( n+1 \) is odd, and we compute:
\begin{align*}
b_{n+1}&= b_n + 3\,b_{n-1} + b_{n-2} \\
       &= r_{n} + 3\,(r_{n-1} + 1) + r_{n-2} \quad \text{(by induction hypothesis)} \\
       &= r_{n} + 3r_{n-1} + 3 + r_{n-2} \\
       &= r_{n} + r_{n-1} + 2 + ( 2\,r_{n-1} + r_{n-2}+1 )\\
       &= r_{n} + r_{n-1} + 2 + r_{n} \quad \text{(by the recurrence for $r_n$ )}   \\
       &= (2\,r_{n} + r_{n-1} + 1) + 1 \\
       &= r_{n+1} + 1.
\end{align*}

\medskip

Therefore, by the principle of mathematical induction, the formula holds for all \( n \in \mathbb{N} \):
\[
b_n = 
\begin{cases}
r_{n}       & \text{if } n \text{ is even}, \\
r_{n}+1   & \text{if } n \text{ is odd}.
\end{cases}
\]
\end{proof}

\begin{rem}
For the Pell and Pell-Lucas numbers, the Pell-Lucas numbers $\widehat{Q}_n$, defined by 
\[
\widehat{Q}_n = 2 \widehat{Q}_{n-1} + \widehat{Q}_{n-2}, \quad \widehat{Q}_0 = \widehat{Q}_1 = 2,
\]
satisfy the identity
\begin{equation}\label{eq:lucas}
    \widehat{Q}_n = E_{n+1} + E_{n-1},
\end{equation}
where \(E_n\) is the \(n\)th Pell number. As noted in \cite{Serkland1972}, no proof is given there, so we provide one here.
\end{rem}

\begin{lem}\label{l:12} 
Let \( E_n \) be the Pell numbers defined by the recurrence
\[
E_n = 2E_{n-1} + E_{n-2}, \quad E_0 = 0,\ E_1 = 1,\ E_2 =2 ,
\]
and let \( \widehat{Q}_n \) be the  Pell–Lucas numbers defined by
\[
 \widehat{Q}_n= 2 \, \widehat{Q}_{n-1} + \widehat{Q}_{n-2} \, \quad \widehat{Q}_0 = 2 ,\,  \widehat{Q}_1 = 2 .
\]
Then, for all \( n \geq 1 \),
\[
 \widehat{Q}_n = E_{n+1} + E_{n-1}. 
\]
\end{lem}

\begin{proof}
We will use the principle of mathematical induction (PMI). When $n = 1$,
\[
\widehat{Q}_1 = E_2 + E_{0}=2+0=2.
\] 
so the result is true. We assume it is true for any positive integer $n=k$:
\[
\widehat{Q}_k = E_{k+1} + E_{k-1}.
\] 
Now, we show that it is true for $n = k + 1$. Using the recurrence relation for \( E_n \), we compute:
\begin{align*}
\widehat{Q}_{k+1} &= 2\widehat{Q}_k + \widehat{Q}_{k-1} \\
        &= 2(E_{k+1}+E_{k-1}) + (E_{k} + E_{k-2}) \quad \text{(by the induction hypothesis)} \\
        &= (2E_{k+1} +  E_{k}) + ( 2E_{k-1}+E_{k-2} ) \\
        &= E_{k+2} +  E_{k} \quad \text{(by the recurrence for $E_{n}$)}              
\end{align*}
Thus, the identity holds for \( n = k+1 \), and by induction, it holds for all positive integers \( n \).
\end{proof}

\begin{cor}\label{c:13} 
Let \( \widehat{Q}_n \) be the Pell-Lucas numbers defined by the recurrence
\[
\widehat{Q}_n = 2 \, \widehat{Q}_{n-1} + \widehat{Q}_{n-2}, \quad \widehat{Q}_0 = 2,\ \widehat{Q}_1 = 2,
\]
and let \( r_n \) be the sequence defined by
\[
r_n = 2 \, r_{n-1} + r_{n-2} + 1, \quad r_0 =r_1 = 0,\ r_2 = 1,
\]
Then, for all \( n \geq 0 \),
\[
\widehat{Q}_n =4\,r_{n} + 2 .
\]

\end{cor}

\begin{proof} 

Let \( b_n \) be the sequence defined by
\[
b_n = b_{n-1} + 3b_{n-2} + b_{n-3}, 
\quad b_0 = 0,\ b_1 = 1,\ b_2 = 1.
\]
Using the identity for Pell-Lucas numbers (see Lemma~\ref{l:12}), we have
\begin{align*}
\widehat{Q}_n &= E_{n+1} + E_{n-1} \\
    &   \quad \text{(by the identity for Pell numbers; see Lemma~\ref{l:4})}\\
&= \big(b_{n+1} + b_{n}\big) + \big(b_{n-1} + b_{n-2}\big). 
\end{align*}       
Next, applying the formula relating the recurrences \( b_{n} \) and \( r_{n} \) from Lemma~\ref{l:11}, we proceed as follows.

\medskip
\noindent

\textbf{Case 1:} \(n\) even.      
\begin{align*}       
\widehat{Q}_n &= b_{n+1} + b_{n} + b_{n-1} + b_{n-2} \\
    &= (r_{n+1} + 1) + r_{n} + (r_{n-1} + 1) + r_{n-2} 
       \quad \text{(by the identity for $b_{n}$; see Lemma~\ref{l:11})} \\
    &= r_{n+1} + r_{n} + r_{n-1} + r_{n-2} + 2\\
    &= r_{n+1} + (2r_{n-1}+ r_{n-2}+1) + r_{n-1} + r_{n-2} + 2\\    
    &= r_{n+1} + 3r_{n-1} + 2r_{n-2} + 3\\ 
    &= (2r_{n}+r_{n-1}+1) + 3r_{n-1} + 2r_{n-2} + 3\\ 
    &= 2r_{n}+4r_{n-1} + 2r_{n-2} + 4\\     
    &= 2(r_{n}+2r_{n-1} + r_{n-2} + 1)+2\\   
    &= 2(r_{n}+r_{n})+2\\      
    &= 4r_{n}+2\\
\end{align*}

\medskip
\noindent
\textbf{Case 2:} \(n\) odd.
\begin{align*}       
\widehat{Q}_n &= b_{n+1} + b_{n} + b_{n-1} + b_{n-2} \\
    &= r_{n+1} + (r_{n} + 1) + r_{n-1} + (r_{n-2} + 1) 
       \quad \text{(by the identity for \(b_{n}\); see Lemma~\ref{l:11})} \\
    &= r_{n+1} + r_{n} + r_{n-1} + r_{n-2} + 2.
\end{align*}
\medskip
In both cases, we obtain
\[
\widehat{Q}_n = 4\, r_{n} +  2,
\]
for all \( n \geq 0 \).

\end{proof}

\begin{cor}\label{c:16}
Let \( Q_n \) be the generalized Pell sequence defined by the recurrence
\[
Q_n = 2 \, Q_{n-1} + Q_{n-2}, \quad Q_0 = 1,\ Q_1 = 3,
\]
and let \( r_n \) be the sequence defined by
\[
r_n = 2 \, r_{n-1} + r_{n-2} + 1, \quad r_0 =r_1 = 0,\ r_2 = 1,
\]
Then, for all \( n \geq 0 \),
\[
Q_n = 2\, r_{n+1} + 1 .
\]

\end{cor}

\begin{proof}
Let \( b_n \) be the sequence defined by
\[
b_n = b_{n-1} + 3b_{n-2} + b_{n-3}, 
\quad b_0 = 0,\ b_1 = 1,\ b_2 = 1.
\]
Then, using the identity for Pell–Lucas numbers (see Lemma~\ref{l:13}), we have
\begin{align*}
Q_n &= E_{n} + E_{n+1} \\ 
    &= (b_{n} + b_{n-1}) + (b_{n+1} + b_{n}) 
       \quad \text{(by the identity for Pell numbers; see Lemma~\ref{l:4})}.
\end{align*}
\medskip
Applying the relation between the recurrences \( b_n \) and \( r_n \) established in Lemma~\ref{l:11}, we consider two cases according to the parity of \( n \):

\medskip

\noindent\textbf{Case 1:} \( n \) even.
\begin{align*}       
   Q_n &= (b_{n} + b_{n-1}) + (b_{n+1} + b_{n}) \\
       &= \big(r_{n} + (r_{n-1} + 1)\big) + \big((r_{n+1} + 1) + r_{n}\big)\\ 
       &\quad \text{(by the identity for \( b_n \); see Lemma~\ref{l:11})} \\
       &= r_{n+1} + 2r_{n} + r_{n-1} + 2\\
       &= r_{n+1} + \big(2r_{n} + r_{n-1} + 1\big) + 1\\    
       &= r_{n+1} + r_{n+1} + 1\\           
       &= 2\,r_{n+1} + 1 .
\end{align*}

\medskip

\noindent\textbf{Case 2:} \( n \) odd.
\begin{align*}       
   Q_n &= (b_{n} + b_{n-1}) + (b_{n+1} + b_{n}) \\
       &= \big((r_{n} + 1) + r_{n-1}\big) + \big(r_{n+1} + (r_{n} + 1)\big)\\ 
         & \quad \text{(by the identity for \( b_n \); see Lemma~\ref{l:11})} \\
       &= r_{n+1} + 2r_{n} + r_{n-1} + 2\\
       &= r_{n+1} + \big(2r_{n} + r_{n-1} + 1\big) + 1\\    
       &= r_{n+1} + r_{n+1} + 1\\           
       &= 2\,r_{n+1} + 1 .
\end{align*}
\medskip
Therefore, for all \( n \geq 0 \), we obtain the identity
\[
Q_n = 2\,r_{n+1} + 1.
\]
\end{proof}
As noted in \cite{Horadam1994}, the identity
\[
\widehat{Q}_n = 2\,Q_{n-1}
\]
holds. Since no proof is given in that reference, we supply one here.
\begin{cor}
Let \( Q_n \) be the generalized Pell sequence defined by the recurrence
\[
Q_n = 2 \, Q_{n-1} + Q_{n-2}, \quad Q_0 = 1,\ Q_1 = 3,
\]
and let \( \widehat{Q}_n \) be the Pell-Lucas numbers defined by
\[
\widehat{Q}_n = 2 \, \widehat{Q}_{n-1} + \widehat{Q}_{n-2}, \quad \widehat{Q}_0 = 2,\ \widehat{Q}_1 = 2,
\]
Then, for all \( n \geq 0 \),
\[
\widehat{Q}_n = 2\,Q_{n-1} .
\]

\end{cor}

\begin{proof}
Let us consider
\begin{align*}       
   \widehat{Q}_n &= 4\,r_{n} + 2 \quad \text{(by the identity for $ \widehat{Q}_n $; see Corollary~\ref{c:13})}  \\
       &= 2\,(2\,r_{n} + 1) \\  
       &= 2\,Q_{n-1} \quad \text{(by the identity for $Q_n $; see Corollary~\ref{c:16})}
\end{align*}
Therefore, for all \( n \geq 0 \), we obtain the identity
\[
\widehat{Q}_n = 2\,Q_{n-1} .
\]
\end{proof}

\begin{cor}\label{l:13a} 
Let \( E_n \) be the Pell numbers defined by the recurrence
\[
E_n = 2 \, E_{n-1} + E_{n-2}, \quad E_0 = 0  ,\ E_1 = 1 ,
\]
and let \( r_n \) be the sequence defined by
\[
r_n = 2 \, r_{n-1} + r_{n-2} + 1, \quad r_1 = 0,\ r_2 = 1,
\]
Then, for all \( n \geq 2 \),
\[
E_n =r_{n}+r_{n-1}+1 .
\]

\end{cor}

\begin{proof}

Let \( b_n \) be the sequence defined by
\[
b_n = b_{n-1} + 3b_{n-2} + b_{n-3}, \quad b_0 = 0,\ b_1 = 1,\ b_2 = 1.
\]
Then
\begin{align*}
E_n &= b_{n} + b_{n-1} \quad \text{(by the identity  Pell-numbers; see Lemma~\ref{l:4})} \\
    &= r_{n}+r_{n-1}+1 \text{(by the identity for $b_{n}$; see Lemma~\ref{l:11})}.
\end{align*}       

\end{proof}

\subsection{Inverse Powers of the Generating Matrix and Identities for Recurrence Relations}

\begin{lem}\label{l:17} 
For all \( n \geq 1 \), we have the following five identities: 
\begin{align}  
 r_n + E_{n}  & =r_{n+1}, \label{id:3} \\
 E_n + Q_{n-1}& =E_{n+1,}  \label{id:4} \quad\textnormal{(appears in \cite{Horadam1994})} \\
2E_n + Q_{n-1} &=Q_{n},  \label{id:5}\\
 E_n + Q_{n-1}- r_n  &=Q_{n}-r_{n+1}. \label{id:6}
\end{align}
\end{lem}

\begin{proof}
Let \( r_n \) be the sequence defined by
\[
r_n = 2 \, r_{n-1} + r_{n-2} + 1, \quad r_0 = r_1 = 0,\ r_2 = 1,
\] and let \( E_n \) denote the Pell numbers, defined by the recurrence relation
\[
E_n = 2E_{n-1} + E_{n-2}, \quad E_0 = 0,\ E_1 = 1,\ E_2 = 2.
\]
Then,
\begin{align*} 
 r_n + E_{n} 
    &= r_n + \big(r_n + r_{n-1} + 1\big) 
       \quad \text{(by the identity for $E_{n}$; see Corollary~\ref{l:13a})}\\ 
    &= 2r_n + r_{n-1} + 1\\ 
    &= r_{n+1}.
\end{align*}
Hence, we arrive at formula ~\ref{id:3}, \,  \( r_n + E_{n} = r_{n+1} \) for all \( n \geq 1 \).
\bigskip

Next, let \( Q_n \) represent the generalized Pell sequence, defined by
\[
Q_n = 2Q_{n-1} + Q_{n-2}, \quad Q_0 = 1,\ Q_1 = 3.
\]
It follows that
\begin{align*} 
 E_n + Q_{n-1}  
    &= E_n + \big(E_{n-1} + E_{n}\big) 
       \quad \text{(by the identity for $Q_{n}$; see Lemma~\ref{l:13})} \\
    &= 2E_{n} + E_{n-1} \\
    &= E_{n+1} \quad \text{(by the recurrence for $E_{n}$)}.
\end{align*}
Therefore, we obtain formula ~\ref{id:4}, \, \( E_n + Q_{n-1} = E_{n+1} \) for all \( n \geq 1 \).  Moreover,
\begin{align*} 
 2E_n + Q_{n-1} 
    &= E_n + \big(E_{n} + Q_{n-1}\big)\\
    &= E_{n} + E_{n+1} 
       \quad \text{(by identity~\ref{id:4})}\\
    &= Q_{n} 
       \quad \text{(by the identity for $Q_{n}$; see Lemma~\ref{l:13})}.
\end{align*}
Thus, we obtain formula ~\ref{id:5}, \, \( 2E_n + Q_{n-1} = Q_{n} \) for every \( n \geq 1 \).

\bigskip 
Finally,
\begin{align*} 
Q_{n} - r_{n+1}      
    &= \big(E_{n} + E_{n+1}\big) - r_{n+1} 
       \quad \text{(by the identity for $Q_{n}$; see Lemma~\ref{l:13})} \\
    &= \big(E_{n} + E_{n+1}\big) - r_{n} - E_{n} 
       \quad \text{(by identity~\ref{id:3})}\\
    &= E_{n+1} - r_{n}\\
    &= E_{n} + Q_{n-1} - r_{n} 
       \quad \text{(by identity~\ref{id:4})}.
\end{align*}
It follows that formula ~\ref{id:6} holds, \( E_{n} + Q_{n-1} - r_{n} = Q_{n} - r_{n+1} \) for all \( n \geq 1 \).

\end{proof}

\begin{cor}\label{l:22} 
Let \( Q_n \) be the generalized Pell sequence defined by the recurrence
\[
Q_n = 2 \, Q_{n-1} + Q_{n-2}, \quad Q_0 = 1,\ Q_1 = 3,
\]
and let \( b_n \) be the sequence defined by
\[
b_n = b_{n-1} + 3b_{n-2} + b_{n-3}, \quad b_0 = 0,\ b_1 = 1,\ b_2 = 1.
\]
Then, for all \( n \geq 1 \),
\[
Q_{n-1} = b_{n+1} - b_{n-1}.
\]
\end{cor}
\begin{proof}
Since
\[
E_{n} = b_{n} + b_{n-1}. \quad \text{(see Lemma~\ref{l:4})}
\]
Substituting this into identity~\ref{id:4},
\[
E_{n} + Q_{n-1} = E_{n+1}, 
\quad \text{(see Lemma~\ref{l:17})}
\]
gives
\[
\begin{aligned}
 b_{n} + b_{n-1} + Q_{n-1} &=  b_{n+1} + b_{n} \\
\end{aligned}
\]
Thus, for all  \( n \geq 1 \), we obtain the identity
\[
Q_{n-1} = b_{n+1} - b_{n-1}.
\]
\end{proof}

\begin{lem}
Let \( u \) be the Pell matrix defined by
\[
u = \begin{pmatrix} 
                    0   &  1   &  1 \\
                    1   &  0   &  1 \\
                    1   &  1   &  1\\
\end{pmatrix}.
\] Then,  the inverse of the matrix \( u^n \) is:

\[
u^{-n} = 
\begin{pmatrix} \vspace{0.2cm}
(-1)^{n}(Q_{n-1}-r_{n})  & (-1)^{n}r_{n} &  (-1)^{n-1}E_{n} \\ \vspace{0.2cm}
 (-1)^{n}r_{n}  & (-1)^{n}(Q_{n-1}-r_{n}) &  (-1)^{n-1}E_{n} \\
(-1)^{n-1}E_{n}  &  (-1)^{n-1}E_{n} & (-1)^{n}Q_{n-1}
\end{pmatrix}
\]
Where \( Q_0 = 1 \), \( Q_1 = 3 \), and \( Q_n = 2 Q_{n-1} + Q_{n-2} \); this defines the generalized Pell sequence.  Moreover, \( r_n = 2 r_{n-1} + r_{n-2} + 1 \), with initial conditions \( r_0 = r_1 = 0 \) and \( r_2 = 1 \).
\end{lem}

\begin{proof}
By the principle of mathematical induction, When $n = 1$, 
\[
u^{-1} = 
\begin{pmatrix}
-1   &  0 &  1 \\
 0   & -1 &  1 \\
 1   &  1 & -1
\end{pmatrix}
\]
so the result is true. Now, assume it is true for an arbitrary positive integer $n$: 
\[
u^{-n} = 
\begin{pmatrix} \vspace{0.2cm}
(-1)^{n}(Q_{n-1}-r_{n})  & (-1)^{n}r_{n} &  (-1)^{n-1}E_{n} \\ \vspace{0.2cm}
 (-1)^{n}r_{n}  & (-1)^{n}(Q_{n-1}-r_{n}) &  (-1)^{n-1}E_{n} \\
(-1)^{n-1}E_{n}  &  (-1)^{n-1}E_{n} & (-1)^{n}Q_{n-1}
\end{pmatrix}
\]
Then we compute the product \( u^{-n}u^{-1} \) using the given matrices:
\[
u^{-n}u^{-1}  = 
\begin{pmatrix} \vspace{0.2cm}
(-1)^{n}(Q_{n-1}-r_{n})  & (-1)^{n}r_{n} &  (-1)^{n-1}E_{n} \\ \vspace{0.2cm}
 (-1)^{n}r_{n}  & (-1)^{n}(Q_{n-1}-r_{n}) &  (-1)^{n-1}E_{n} \\
(-1)^{n-1}E_{n}  &  (-1)^{n-1}E_{n} & (-1)^{n}Q_{n-1}
\end{pmatrix}
\begin{pmatrix}
-1   &  0 &  1 \\
 0   & -1 &  1 \\
 1   &  1 & -1
\end{pmatrix}
\]
Carrying out the matrix multiplication, we obtain:
{\small
\[
u^{-n}u^{-1}  = 
\begin{pmatrix} \vspace{0.2cm}
(-1)^{n+1}(E_{n}+Q_{n-1}-r_{n})  & (-1)^{n+1}(r_{n}+E_{n}) &  (-1)^{n}(E_{n}+Q_{n-1}) \\ \vspace{0.2cm}
 (-1)^{n+1}(r_{n}+E_{n})  & (-1)^{n+1}(E_{n}+Q_{n-1}-r_{n}) &  (-1)^{n}(E_{n}+Q_{n-1}) \\
(-1)^{n}(E_{n}+Q_{n-1})  &  (-1)^{n}(E_{n}+Q_{n-1}) & (-1)^{n+1}(2E_{n}+Q_{n-1})
\end{pmatrix}
\]} 
Recall that the sequences \( (E_n) \), \( (Q_n) \) and \( (r_n) \) satisfies the following relations, see Lemma~\ref{l:17}.
\begin{align*}       
 r_n + E_{n}  & =r_{n+1}, \\
 E_n + Q_{n-1}& =E_{n+1,} \\
 2E_n + Q_{n-1} &=Q_{n},  \\
 E_{n}+Q_{n-1}-r_{n} &=Q_{n}-r_{n+1}. 
\end{align*}
Substituting into the matrix, we get:

\[
u^{-n}u^{-1}  = 
\begin{pmatrix} \vspace{0.2cm}
(-1)^{n+1}(Q_{n}-r_{n+1})  & (-1)^{n+1}r_{n+1} &  (-1)^{n}E_{n+1} \\ \vspace{0.2cm}
 (-1)^{n+1}r_{n+1}  & (-1)^{n+1}(Q_{n}-r_{n+1}) &  (-1)^{n}E_{n+1} \\
(-1)^{n}E_{n+1}  &  (-1)^{n}E_{n+1} & (-1)^{n+1}Q_{n}
\end{pmatrix}
\]
Hence,
\[
u^{-n}u^{-1} = u^{-(n+1)}
\]
This completes the proof.
\end{proof}

\begin{cor}\label{c:23} 
For all positive integer \( n \geq 1 \), following equalities hold:
\begin{align}
Q^2_{n-1} - (b_n + r_n)Q_{n-1}+2b_n r_n - E_n^2 &= (-1)^{n},\label{id:15} \\
E_n^2 -(b_n + r_n) Q_{n-1}+2b_n r_n  &= 0,\label{id:16} \\
2E_n^2 - Q^2_{n-1}                   &= (-1)^{n-1}.\label{id:17}  
\end{align}

\end{cor}

\begin{proof}
Let \( U \) be the Pell matrix defined by
\[
U = \begin{pmatrix} 
                    0   &  1   &  1 \\ 
                    1   &  0   &  1 \\ 
                    1   &  1   &  1 
\end{pmatrix}.
\]
Then, the inverse of the matrix \( U^n \) is given by
\[
U^{-n} = 
\begin{pmatrix} \vspace{0.2cm}
(-1)^{n}(Q_{n-1}-r_{n})  & (-1)^{n}r_{n} &  (-1)^{n-1}E_{n} \\ \vspace{0.2cm}
(-1)^{n}r_{n}  & (-1)^{n}(Q_{n-1}-r_{n}) &  (-1)^{n-1}E_{n} \\ 
(-1)^{n-1}E_{n}  &  (-1)^{n-1}E_{n} & (-1)^{n}Q_{n-1}
\end{pmatrix}.
\]
Here \( Q_0 = 1 \), \( Q_1 = 3 \), and \( Q_n = 2 Q_{n-1} + Q_{n-2} \), which is the generalized Pell sequence.  
Moreover, \( r_n \) satisfies the recurrence \( r_n = 2 r_{n-1} + r_{n-2} + 1 \), with initial conditions \( r_0 = r_1 = 0 \) and \( r_2 = 1 \).

\quad

On the other hand, we have defined \( U^n \) as follows: 

\[
U^{n} =  \begin{pmatrix} \vspace{0.2cm}
                   E_{n}+E_{n-1}- b_{n} & b_{n} & E_{n} \\ \vspace{0.2cm}
                   b_{n} & E_{n}+E_{n-1}-b_{n} & E_{n} \\ 
                   E_{n} & E_{n} & E_{n}+E_{n-1} 
\end{pmatrix}.
\]   

\quad

The product \( U^n U^{-n} \) yields the following \( 3 \times 3 \) matrix:

{\small
\[
U^{n}U^{-n} =
\begin{bmatrix}
A  & B  & C \\ 
B  & A  & C \\ 
C  & C  & D
\end{bmatrix},
\]
}
where
\[
\begin{aligned}
A &= (-1)^n (E_n + E_{n-1} - b_n)(Q_{n-1} - r_n) + (-1)^n b_n r_n + (-1)^{n-1} E_n^2, \\ 
B &= (-1)^{n-1} E_n^2 + (-1)^n (E_n + E_{n-1} - b_n)r_n + (-1)^n (Q_{n-1} - r_n)b_n, \\ 
C &= (-1)^{n-1}(E_n + E_{n-1} - Q_{n-1})E_n , \\ 
D &= 2(-1)^{n-1} E_n^2 + (-1)^n (E_n + E_{n-1}) Q_{n-1}.
\end{aligned}
\]
On the other hand, we know that
\[
U^n U^{-n} = 
 \begin{pmatrix} \vspace{0.2cm}
                  1  & 0 & 0 \\ \vspace{0.2cm}
                  0  & 1 & 0 \\ 
                  0  & 0 & 1 
\end{pmatrix}.
\]
Therefore, by equating the two matrices obtained through matrix multiplication and applying the relation from Lemma~\ref{l:13},  
\[
E_{n} + E_{n-1} = Q_{n-1},
\]  
we establish the identities stated in the corollary.
\end{proof}

\subsection{Diagonalization of the Generating Matrix and Binet’s Formula}

\begin{thm}
Let $n$ be an integer. The  Binet formula of the sequence $b_{n}$ is
\begin{align}
b_{n} &= \dfrac{(-1)^{n+1}}{2}+\dfrac{1}{4}\left[ (1+\sqrt{2})^{n} + (1-\sqrt{2})^{n}\right], \\
b_{n}  &= \dfrac{(-1)^{n+1}+Q_{n-1}}{2}, \\ 
E_{n}+E_{n-1}-b_{n} &= \dfrac{(-1)^{n}+Q_{n-1}}{2} . 
\end{align}
\end{thm}

\begin{proof}
Let the matrix $U$ be as in Lemma~\ref{l:5}. If we calculate the eigenvalues and eigenvectors of the matrix $U$ are
\[
\lambda_1=-1, \quad \lambda_2=1-\sqrt{2}, \quad \lambda_3=1+\sqrt{2}
\]
and

\[v_1=
\begin{pmatrix}
-1 \\ \\
 1 \\  \\
 0
\end{pmatrix}, \, v_2=
\begin{pmatrix}
-\dfrac{\sqrt{2}}{2} \\ \\
-\dfrac{\sqrt{2}}{2} \\  \\
1
\end{pmatrix}, \, v_3=
\begin{pmatrix}
\dfrac{\sqrt{2}}{2} \\ \\
\dfrac{\sqrt{2}}{2} \\ \\
1
\end{pmatrix}
\]
respectively. Then, we can diagonalize of the matrix $U$ by
\[
D = P^{-1}UP
\]
where
\[
P = (v_1, v_2, , v_3) = 
\begin{pmatrix}
-1 & \dfrac{\lambda_2-1}{2} & \dfrac{\lambda_3-1}{2}  \\ \\
1 & \dfrac{\lambda_2-1}{2} & \dfrac{\lambda_3-1}{2}  \\ \\
0 & 1 & 1
\end{pmatrix},
\quad 
P^{-1} = 
\begin{pmatrix}
-\dfrac{1}{2}  & \dfrac{1}{2}  & 0 \\ \\ 
\dfrac{\lambda_2-1}{4} & \dfrac{\lambda_2-1}{4} & \dfrac{1}{2} \\ \\ \\
\dfrac{\lambda_3-1}{4}  & \dfrac{\lambda_3-1}{4} & \dfrac{1}{2}
\end{pmatrix}
\]
and
\[
D = \text{diag}(\lambda_1,\lambda_2,,\lambda_3) =
\begin{pmatrix}
-1 & 0 & 0 \\
0 & 1-\sqrt{2} & 0 \\
0 & 0 & 1+\sqrt{2}
\end{pmatrix}.
\]
From the properties of the similar matrices, we can write 
\begin{equation}\label{i:37}
U^n = P D^n P^{-1}
\end{equation}
where $n$ is any integer and
\[
D^n = 
\begin{pmatrix}
(-1)^{n} & 0 & 0 \\  \vspace{0.2cm}
0 & (1-\sqrt{2})^{n} & 0 \\
0 & 0 & (1+\sqrt{2})^{n}
\end{pmatrix}=\begin{pmatrix}
(-1)^{n} & 0 & 0 \\
0 & Q_{n-1}-E_{n}\sqrt{2} & 0 \\
0 & 0 & Q_{n-1} + E_{n}\sqrt{2} 
\end{pmatrix}.
\]
By equation~\ref{i:37}, we get
\[
U^{n}
=
\begin{pmatrix}
E_{n} + E_{n-1} - b_{n} & b_{n} & E_{n} \\
b_{n} & E_{n} + E_{n-1} - b_{n} & E_{n} \\
E_{n} & E_{n} & E_{n} + E_{n-1}
\end{pmatrix}
=
\]

\[
\renewcommand{\arraystretch}{1.4}
\begin{bmatrix}
\dfrac{(-1)^{n}+Q_{n-1}}{2} & \dfrac{(-1)^{n+1}+Q_{n-1}}{2} & E_{n} \\[0.2cm]
\dfrac{(-1)^{n+1}+Q_{n-1}}{2} & \dfrac{(-1)^{n}+Q_{n-1}}{2} & E_{n} \\[0.2cm]
E_{n} & E_{n} & Q_{n-1}
\end{bmatrix}
\]

\[
=
\begin{pmatrix}
\dfrac{(-1)^{n}}{2}+\dfrac{1}{4}\!\left(\lambda_2^{n}+\lambda_3^{n}\right) 
& \dfrac{(-1)^{n+1}}{2}+\dfrac{1}{4}\!\left(\lambda_2^{n}+\lambda_3^{n}\right) 
& \dfrac{1}{2\sqrt{2}}\!\left( \lambda_3^{n} - \lambda_2^{n}  \right) \\[0.4cm]
\dfrac{(-1)^{n+1}}{2}+\dfrac{1}{4}\!\left(\lambda_2^{n}+\lambda_3^{n}\right) 
& \dfrac{(-1)^{n}}{2}+\dfrac{1}{4}\!\left(\lambda_2^{n}+\lambda_3^{n}\right) 
& \dfrac{1}{2\sqrt{2}}\!\left( \lambda_3^{n} - \lambda_2^{n}  \right) \\[0.4cm]
\dfrac{1}{2\sqrt{2}}\!\left( \lambda_3^{n} - \lambda_2^{n}  \right) 
& \dfrac{1}{2\sqrt{2}}\!\left( \lambda_3^{n} - \lambda_2^{n}  \right) 
& \dfrac{1}{2}\!\left( \lambda_3^{n} + \lambda_2^{n}  \right)
\end{pmatrix}.
\]
Thus, the proof is completed. 

\end{proof}

\section{Recurrences Generated by Binary $3 \times 3$ Matrices with Determinant Minus One}
Consider the generating matrix
\[
U = \begin{pmatrix}1 & 0 & 1\\ 0 & 1 & 1\\ 1 & 1 & 1\end{pmatrix}.
\]
In this section, we analyze the Pell \textit{U}-matrix with determinant \(-1\), which provides a matrix representation of Pell numbers. This matrix is employed to compute powers \(U^n\), determinants, inverses, and Cassini-type identities; to derive the characteristic roots and the Binet formulas for the sequences \(a_n\) and \(r_n\); to establish identities involving the Pell sequence in conjunction with the generalized Pell sequence and \(b_n\); and to deduce summation formulas for the Pell numbers as well as for the sequences \(a_n\) and \(r_n\).

\subsection{The Matrix Representation of Recurrences and Their Identities}

\begin{lem}\label{l:24} 
Let \(a_n \) be the sequence defined by the recurrence relation
\[
a_n = 3a_{n-1} - a_{n-2} - a_{n-3}, 
\quad a_0 = 1, \ a_1 = 1, \ a_2 = 2,
\]
and let \(r_n \) be the sequence defined by
\[
r_n = 2r_{n-1} + r_{n-2} + 1, 
\quad r_0 = 0, \ r_1 = 0, \ r_2 = 1.
\]
Then, for all \( n \geq 0 \),
\[
a_{n} = r_{n} + 1.
\]
\end{lem}

\begin{proof}
We will use the principle of mathematical induction (PMI). When $n = 0$,
\[
a_{0} = r_{0} + 1=1.
\]
so the result is true. We assume it is true for any positive integer $n = k$:
\[
a_{k} = r_{k} + 1.
\] 
Now, we show that it is true for $n = k + 1$. Using the recurrence relation for \(a_{k}\), we can write
\begin{align*}
a_{k+1} &= 3a_{k} - a_{k-1} - a_{k-2} \\   
        &= 3(r_{k}+1) - (r_{k-1}+1) - (r_{k-2}+1) \quad\text{(by the induction hypothesis)} \\   
        &= 3r_{k} - r_{k-1} - r_{k-2}+1 \\ 
        &= 3(2r_{k-1} + r_{k-2} + 1) - r_{k-1} - r_{k-2}+1 \quad\text{(by the recurrence for $r_{k}$) }\\  
        &= 5r_{k-1} + 2r_{k-2} + 4 \\   
        &= (4r_{k-1} + 2r_{k-2} + 2) + r_{k-1}+2 \\  
        &= 2(2r_{k-1} + r_{k-2} + 1) + r_{k-1}+2 \\ 
        &= 2 r_{k} + r_{k-1}+ 2  \quad\text{(by the recurrence for $r_{k}$) } \\ 
        &= (2 r_{k} + r_{k-1}+ 1) + 1 \\                                     
        &= r_{k+1} + 1.                                
\end{align*}
and the result follows by the principle of mathematical induction.
\end{proof}

\begin{lem}\label{l:25} 
Let \( a_n \) be the numbers defined by the recurrence relation
\[
a_n = 3 \, a_{n-1} - a_{n-2} - a_{n-3}, \quad a_0 = 1,\ a_1 = 1, \ a_2 = 2,
\]
and let  \( E_n \) be the Pell numbers defined by the recurrence relation  
\[
E_n = 2E_{n-1} + E_{n-2},
\]
with initial conditions  \( E_0 = 0 \) and \( E_1 = 1 \)
Then, for all \( n \geq 0 \),

\[
a_{n+1} = a_{n} + E_{n} 
\]

\end{lem}

\begin{proof}
For all \( n \geq 0 \), we have
\begin{align*}
a_{n} + E_{n} &= (r_{n} + 1) + E_{n} \quad\text{(by the identity for $a_{n}$ ; see Lemma ~\ref{l:24}) )} \\   
                  &= (r_{n} + E_{n}) + 1  \\   
                  &= r_{n+1} + 1 \quad \text{(by the identity ~\ref{id:3}; see Lemma~\ref{l:17})}\\                                                             
                  &= a_{n+1}.\quad\text{(by the identity for $a_{n}$ ; see Lemma ~\ref{l:24}) )}
\end{align*}
Therefore, the formula holds for all \( n \in \mathbb{N} \).
\end{proof}

\begin{cor}\label{c:26} 

Let \( r_n \) be the sequence defined by  
\[
r_n = 2r_{n-1} + r_{n-2} + 1, 
\quad r_0 = 0, \ r_1 = 0, \ r_2 = 1.
\]

Let \( E_n \) denote the Pell sequence defined by the recurrence relation  
\[
E_n = 2E_{n-1} + E_{n-2},
\]
with initial conditions \( E_0 = 0 \) and \( E_1 = 1 \).  

Similarly, let \( a_n \) be the sequence defined by  
\[
a_n = 3a_{n-1} - a_{n-2} - a_{n-3},
\]
with  \( a_0 = 1 \), \( a_1 = 1 \), \( a_2 = 2 \), and \( a_3 = 4 \).  

Then, for all \( n \geq 0 \),
\[
E_{n+1} = E_{n} + r_{n} + a_{n}.
\]
\end{cor}

\begin{proof}
Observe that
\begin{align*}
E_{n} + r_{n}+ a_{n} &= (E_{n} + r_{n})+ a_{n}  \\   
                  &= r_{n+1}+ a_{n} \quad \text{(by the identity ~\ref{id:3} for $r_{n+1}$; see Lemma~\ref{l:17})}  \\   
                  &= r_{n+1}+ r_{n} + 1 \quad\text{(by the identity for $a_{n}$ ; see Lemma ~\ref{l:24}) )}\\                                                             
                  &= E_{n+1}\quad \text{(by the identity for $E_{n}$; see Corollary ~\ref{l:13a})}.
\end{align*}
Therefore, the formula holds for all \( n \in \mathbb{N} \).
\end{proof}
The following lemma can be proven by mathematical induction.
\begin{lem}\label{l:27} 
Let $u$ be the symmetric matrix
\[
u = \begin{pmatrix} 
                    1   &  0   &  1 \\
                    0   &  1   &  1 \\
                    1   &  1   &  1\\
\end{pmatrix}.
\]  
Then, 

\[u^{n} =  \begin{pmatrix} \vspace{0.2cm}
                   a_{n}   & r_{n} & E_{n}  \\ \vspace{0.2cm}
                   r_{n}   & a_{n}   & E_{n} \\
                   E_{n}   & E_{n}   & r_{n}+a_{n}  \\
\end{pmatrix}  \]   
where \( n\in \mathbb{Z}^{+} \), and \( E_n \) denotes the Pell numbers, defined by the recurrence relation  
\[
E_n = 2E_{n-1} + E_{n-2},
\]
with initial conditions  \( E_0 = 0 \) and \( E_1 = 1 \). Similarly, \( a_n \) satisfies  
\[
a_n = 3a_{n-1} - a_{n-2} - a_{n-3},
\]
with \( a_0 = 1 \), \( a_1 = 1 \), \( a_2 = 2 \), and \( a_3 = 4 \).  The sequence \( r_n \) is given by  
\[
r_n = 2r_{n-1} + r_{n-2} + 1,
\]
with \( r_0 = 0 \), \( r_1 = 0 \) and \( r_2 = 1 \).

\end{lem}

\begin{proof}
We will use the principle of mathematical induction (PMI). When $n = 2$,
\[
u^{2} = \begin{bmatrix}  
                   a_{2}   & r_{2}   & E_{2}  \\ \vspace{0.2cm}
                   r_{2}   & a_{2}   & E_{2} \\
                   E_{2}   & E_{2}   & r_{2}+a_{2}  \\
\end{bmatrix} 
= 
\begin{bmatrix} 
2 & 1 & 2\\ 
1 & 2 & 2\\  
2 & 2 & 3\\ 
\end{bmatrix}
\]
so the result is true. We assume it is true for any positive integer $n = k$:

\[u^{k} =  \begin{pmatrix} \vspace{0.2cm}
                   a_{k}   & r_{k}   & E_{k}  \\ \vspace{0.2cm}
                   r_{k}   & a_{k}   & E_{k} \\
                   E_{k}   & E_{k}   & r_{k}+a_{k}  \\
\end{pmatrix}\]  
Now, we show that it is true for $n = k + 1$. Then, we can write
\begin{align*}
u^{k+1} = u^k u 
&= \begin{pmatrix} \vspace{0.2cm}
                   a_{k}   & r_{k}   & E_{k}  \\ \vspace{0.2cm}
                   r_{k}   & a_{k}   & E_{k} \\
                   E_{k}   & E_{k}   & r_{k}+a_{k}  \\
\end{pmatrix}
\begin{pmatrix} 
                    1   &  0   &  1 \\
                    0   &  1   &  1 \\
                    1   &  1   &  1\\
\end{pmatrix} \quad\text{(by the induction hypothesis)} \\
&= 
\begin{pmatrix} \vspace{0.2cm}
                   a_{k}+E_{k}         & r_{k}+ E_{k}          & a_{k}+ r_{k}+E_{k}  \\ \vspace{0.2cm}
                   r_{k}+ E_{k}        &  a_{k}+E_{k}          & r_{k}+a_{k}+ E_{k} \\
                   E_{k}+ r_{k}+a_{k}  & E_{k}+ r_{k}+a_{k}    & 2E_{k}+ r_{k}+a_{k} \\
\end{pmatrix}
\end{align*}
Recall that the sequences  satisfies the following relations:  
\begin{align*}
r_{k+1} &= r_{k} + E_{k}  \quad\text{(see Lemma~\ref{l:17})}\\
a_{k+1} &= a_{k} + E_{k}  \quad\text{(see  Lemma~\ref{l:25})}\\
E_{k+1} &= E_{k}+r_{k} +a_{k}  \quad\text{(see Corollary~\ref{c:26})} \\
\end{align*}
Substituting into the matrix, we get:
\begin{align}
u^{k+1} = u^k u &=
\begin{pmatrix} \vspace{0.2cm} 
                   a_{k+1}   & r_{k+1}   & E_{k+1}  \\ \vspace{0.2cm}
                   r_{k+1}   & a_{k+1}   & E_{k+1} \\
                   E_{k+1}   & E_{k+1}   & r_{k+1}+a_{k+1}  \\
\end{pmatrix}
\end{align}
and the result follows by the principle of mathematical induction.
\end{proof}

\begin{cor}\label{c:28}
Let \( u \) be the Pell matrix defined by
\[
u = \begin{pmatrix} 
                    1   &  0   &  1 \\
                    0   &  1   &  1 \\
                    1   &  1   &  1\\
\end{pmatrix}.
\] Then, for every \( n \in \mathbb{N} \), we have that

\[(a_n  + r_{n})^2 - 2E_n^{2} = (-1)^n.\]

\end{cor}

\begin{proof}
It is easy to see that
\[
\det(u) = -1.
\]
Then, it can be written
\begin{align*}
\det(u^n) &= \det(F) \cdot \det(F) \cdot \dots \cdot \det(F) \\
          &= (-1)^n.
\end{align*}
If \( x = E_n \), \( y = r_{n} \), and \( z = a_n \), then the determinant of the matrix \( u^n \) given in Lemma~\ref{l:5} is

\begin{align*}
\left|
\begin{array}{ccc}
z & y & x \\
y & z & x \\
x & x & y+z
\end{array}
\right|
&= -y^3 + 2x^2y + z^3 + yz^2 - 2x^{2}z - y^{2}z\\
&= z^3-y^3 + (y - z)(2x^{2}-yz)\\
&= a_n ^3-r_{n}^3 + (r_{n} - a_n )(2E_n^{2}-r_{n}a_n )\\
&= a_n ^3-r_{n}^3 - (2E_n^{2}-r_{n}a_n )\\
& \quad \text{(by the identities for $a_n$; see Lemma~\ref{l:24})}\\
&= (a_{n} -r_{n})(a_n ^2 + a_{n}r_n + r_{n}^2) - (2E_n^{2}-r_{n}a_n )\\
&= (a_{n} -r_{n})(a_n ^2 + a_{n}r_n + r_{n}^2) - (2E_n^{2}-r_{n}a_n )\\
&= (a_n ^2 + a_{n}r_n + r_{n}^2) - (2E_n^{2}-r_{n}a_n )\\
&\quad \text{(by the identities for $a_n$; see Lemma~\ref{l:24})}\\
&= (a_n ^2 + 2a_{n}r_n + r_{n}^2) - 2E_n^{2}\\
&= (a_n  + r_{n})^2 - 2E_n^{2}.
\end{align*}
Thus, \[(a_n  + r_{n})^2 - 2E_n^{2} = (-1)^n.\] for all $n\geq 1.$ 
\end{proof}

\begin{cor}\label{c:31}
For $n\geq 1$ and  $m\geq 1$, we have the following identity:
\begin{align}
a_{m+n} & = a_{m}a_{n} + r_{m}r_{n} + E_{m}E_{n},\label{i:20} \\
r_{m+n} & = r_{m}a_{n}+a_{m}r_{n}+E_{m}E_{n} ,\label{i:21}    \\
E_{m+n} & = E_n(r_m + a_{m}) + E_m(r_n + a_{n}) ,\label{i:22} \\
r_{m+n}+a_{m+n} & = 2E_{m} E_n+(r_m + a_{m})(r_n + a_{n}). \label{i:23}
\end{align}
\end{cor}

\begin{proof}
For \( m, n \geq 1 \), we know that \( U^{m+n} = U^m U^n \). Since we have defined \( U^n \) as follows, the same expression in matrix form is:

\[
U^m = \begin{pmatrix} \vspace{0.2cm}
                  a_{m} & r_{m} & E_{m} \\ \vspace{0.2cm}
                  r_{m} & a_{m} & E_{m} \\
                  E_{m} & E_{m} & r_{m}+a_{m} \\
\end{pmatrix}, 
\] 

\[
U^{n} =  \begin{pmatrix} \vspace{0.2cm}
                  a_{n} & r_{n} & E_{n} \\ \vspace{0.2cm}
                  r_{n} & a_{n} & E_{n} \\
                  E_{n} & E_{n} & r_{n}+a_{n}
\end{pmatrix} . \]   
The product \( U^m U^n \) is given by the following symmetric \( 3 \times 3 \) matrix:
\[
U^m U^n =
\begin{bmatrix}
A & B & C \\
B & A & C \\
C & C & D
\end{bmatrix}
\]
where:

\[
\begin{aligned}
A &= a_{m}a_{n} + r_{m}r_{n} + E_{m}E_{n}, \\
B &= r_{m}a_{n}+a_{m}r_{n}+E_{m}E_{n}, \\
C &= E_n(r_m + a_{m}) + E_m(r_n + a_{n}), \\
D &= 2E_{m} E_n+(r_m + a_{m})(r_n + a_{n}).
\end{aligned}
\]
On the other hand, we have:

\[
U^{m+n} = 
\begin{pmatrix} \vspace{0.2cm}
                  a_{m+n} & r_{m+n} & E_{m+n} \\ \vspace{0.2cm}
                  r_{m+n} & a_{m+n} & E_{m+n} \\
                  E_{m+n} & E_{m+n} & r_{m+n}+a_{m+n} \\
\end{pmatrix}. 
\]
Equating the two matrices obtained via matrix multiplication yields the identities stated in the corollary.
\end{proof}

\begin{cor}\label{c:32} 
Let \( a_n \) be the numbers defined by the recurrence relation
\[
a_n = 3 \, a_{n-1} - a_{n-2} - a_{n-3}, \quad a_0 = 1,\ a_1 = 1, \ a_2 = 2,
\]
and let  \( Q_n \) be the generalized Pell sequence defined by the recurrence relation  
\[
Q_n = 2Q_{n-1} + Q_{n-2},
\]
with initial conditions  \( Q_0 = 1 \) and \( Q_1 = 3 \)
Then, for all \( n \geq 0 \),
\begin{align}
a_{n} + r_{n} &= Q_{n-1}, \label{i:24}\\ 
E_{n+1}       &= E_{n} + Q_{n-1}, \label{i:25} \quad\textnormal{(appears in \cite{Horadam1994})} \\
a_{n+1}       &= a_{n} + E_{n}. \label{i:26}
\end{align}
\end{cor}
\begin{proof}
Let \(r_n\) be the sequence defined by
\[
r_n = 2r_{n-1} + r_{n-2} + 1, 
\quad r_0 = 0, \ r_1 = 0, \ r_2 = 1.
\]
For all \(n \geq 0\), we have
\begin{align*}
a_{n} + r_{n} 
    &= (r_{n} + 1) + r_{n} 
       && \text{(by the identity for $a_{n}$; see Lemma~\ref{l:24})} \\   
    &= 2r_{n} + 1 \\   
    &= Q_{n-1} 
       && \text{(by the identity for $Q_{n}$; see Corollary~\ref{c:16})}.
\end{align*}
Therefore, identity~\ref{i:24} holds for all \(n \in \mathbb{N}\). Furthermore,
\begin{align*}
E_{n+1} 
    &= E_{n} + a_{n} + r_{n} 
       && \text{(by the identity for $E_{n+1}$; see Corollary~\ref{c:26})} \\   
    &= E_{n} + Q_{n-1} 
       && \text{(by identity~\ref{i:24} for $Q_{n-1}$)}.
\end{align*}
Thus,
\[
E_{n+1} = E_{n} + Q_{n-1}.
\]
In addition, we obtain
\begin{align*}
a_{n+1} 
    &= a_{n}a_{1} + r_{n}r_{1} + E_{n}E_{1} 
       && \text{(by identity~\ref{i:20} for $a_{n+m}$; see Corollary~\ref{c:31})} \\   
    &= a_{n} + E_{n} 
       && \text{(using identity~\ref{i:24} for $Q_{n-1}$)}.
\end{align*}
Therefore,
\[
a_{n+1} = a_{n} + E_{n}.
\]
\end{proof}

\subsection{Inverse Powers of the Generating Matrix and Identities for Recurrence Relations}

\begin{lem}\label{l:29} 
Let \( E_n \) be the Pell numbers defined by the recurrence
\[
E_n = 2E_{n-1} + E_{n-2}, \quad E_0 = 0,\ E_1 = 1,
\]
and let \( b_n \) be the sequence defined by
\[
b_n = b_{n-1} + 3b_{n-2} + b_{n-3}, \quad b_0 = 0,\ b_1 = 1,\ b_2 = 1.
\]
Then, for all \( n \geq 1 \),
\begin{align}       
 b_{n}+E_{n}      &= (-1)^{n+1} + b_{n+1} \label{i:9},  \\
 2\,b_{n}+E_{n}   &= (-1)^{n+1} + E_{n+1}\label{i:10}.
\end{align}
\end{lem}

\begin{proof}
We will use the principle of mathematical induction (PMI). When $n = 1$,
\[
 b_{1}+E_{1}= 1+1 = (-1)^{2} + b_{2}.
\] 
so the result is true. We assume it is true for any positive integer $n=k$:
\[
 b_{k}+E_{k} = (-1)^{k+1} + b_{k+1}.
\] 
Now, we show that it is true for $n = k + 1$. Using the recurrence relation for \( b_n \), we compute:

\begin{align*}       
(-1)^{n+2} + b_{n+2}  &= (-1)^{n+2} + b_{n+1} + 3b_{n} + b_{n-1}\\
                      &= ((-1)^{n}+b_{n}) + (b_{n+1}+b_{n}) + (b_{n}+b_{n-1})\\
                      &= \big( (-1)^{n} + b_{n} \big) + E_{n+1} + E_{n}\\                       &\quad \text{(by the identity  Pell-numbers; see Lemma~\ref{l:4}))}\\                      
                      &= (-1)^{n}+ (b_{n}+E_{n}) + E_{n+1} \\                      
                      &= (-1)^{n} + ((-1)^{n+1} + b_{n+1}) + E_{n+1} \quad \text{(by the induction hypothesis)}  \\
                      &= b_{n+1}+E_{n+1}.
\end{align*}

Thus, the identity~\ref{i:9} holds for \( n = k+1 \); therefore, by induction, it is valid for all positive integers \( n \).

\medskip 

Now, we have 
\begin{align*}       
2\,b_{n} + E_{n}      &= b_{n}+(b_{n} + E_{n})\\
                      &= b_{n}+((-1)^{n+1}+b_{n+1})\quad \text{(by identity~\ref{i:9})} \\
                      &=(-1)^{n+1} + ( b_{n+1}+b_{n} )\\
                      &=(-1)^{n+1} + E_{n+1}  \quad \text{(by the identity  Pell-numbers; see Lemma~\ref{l:4})}.                     
\end{align*}
Thus, formula~\ref{i:10} is verified. 
\end{proof}

\begin{lem}
Let \( u \) be the Pell matrix defined by
\[
u = \begin{pmatrix} 
                    1   &  0   &  1 \\
                    0   &  1   &  1 \\
                    1   &  1   &  1\\
\end{pmatrix}.
\] Then,  the inverse of the matrix \( u^n \) is:

\[
u^{-n} = 
\begin{pmatrix} \vspace{0.2cm}
 1+(-1)^{n}b_{n} &  (-1)^{n}b_{n}    &  (-1)^{n-1}E_{n} \\ \vspace{0.2cm}
(-1)^{n}b_{n} & 1+(-1)^{n}b_{n}   &  (-1)^{n-1}E_{n} \\
(-1)^{n-1}E_{n}& (-1)^{n-1}E_{n}  &  (-1)^{n}Q_{n-1} \\
\end{pmatrix}.
\]
Where \( Q_{n}  = 2Q_{n-1}  + Q_{n-2} \), \( Q_0 = 1 \), \( Q_1 = 3 \), this is the generalized Pell sequence.  Also, \(b_{n}  = b_{n-1}  + 3b_{n-2} + b_{n-3} \), \, \( b_0 = 0 \), \( b_1 = 1 \), \( b_2 = 1 \).
\end{lem}

\begin{proof}
By the principle of mathematical induction, When $n = 1$, 
\[
u^{-1} = 
\begin{pmatrix}
  0   &  -1 &  1 \\
 -1   &   0 &  1 \\
  1   &   1 & -1
\end{pmatrix}
\]
so the result is true. Now, assume it is true for an arbitrary positive integer $n$: 
\[
u^{-n} = 
\begin{pmatrix} \vspace{0.2cm}
 1+(-1)^{n}b_{n} &  (-1)^{n}b_{n}    &  (-1)^{n-1}E_{n} \\ \vspace{0.2cm}
(-1)^{n}b_{n} & 1+(-1)^{n}b_{n}   &  (-1)^{n-1}E_{n} \\
(-1)^{n-1}E_{n}& (-1)^{n-1}E_{n}  &  (-1)^{n}Q_{n-1} 
\end{pmatrix}
\]
Then we compute the product \( u^{-n}u^{-1} \) using the given matrices:

\[
u^{-n}u^{-1}  = 
\begin{pmatrix} \vspace{0.2cm}
 1+(-1)^{n}b_{n} &  (-1)^{n}b_{n}    &  (-1)^{n-1}E_{n} \\ \vspace{0.2cm}
(-1)^{n}b_{n} & 1+(-1)^{n}b_{n}   &  (-1)^{n-1}E_{n} \\
(-1)^{n-1}E_{n}& (-1)^{n-1}E_{n}  &  (-1)^{n}Q_{n-1} \\
\end{pmatrix}
\begin{pmatrix}
 0   &  -1 &  1 \\
 -1   &  0 &  1 \\
 1   &  1 & -1
\end{pmatrix}
\]
Carrying out the matrix multiplication, we obtain:
{\small
\[
u^{-n}u^{-1}  = 
\begin{pmatrix} \vspace{0.2cm}
(-1)^{n+1}(b_{n}+E_{n})  & -1+(-1)^{n+1}(b_{n}+E_{n})  &  1+(-1)^{n}(2b_{n}+E_{n}) \\ \vspace{0.2cm}
-1+(-1)^{n+1}(b_{n}+E_{n})  & (-1)^{n+1}(b_{n}+E_{n})  &  1+(-1)^{n}(2b_{n}+E_{n}) \\
(-1)^{n}(E_{n}+Q_{n-1}) &  (-1)^{n}(E_{n}+Q_{n-1}) & (-1)^{n+1}(2E_{n}+Q_{n-1})
\end{pmatrix}
\]}
Recall that the sequences \( (E_n) \), \( (Q_n) \), and  \( (b_n) \) satisfy the following relations (see Lemma ~\ref{l:17} and Lemma~\ref{l:29}).
\begin{align*}       
 2\,E_n + Q_{n-1} &= Q_{n}  \\
 E_n + Q_{n-1}  &= E_{n+1} \\
 b_{n}+E_{n}    &= (-1)^{n+1} + b_{n+1}  \\
 2\,b_{n}+E_{n}   &= (-1)^{n+1} + E_{n+1}
\end{align*}
Substituting into the matrix, we get:

\[
u^{-n}u^{-1}  = 
\begin{pmatrix} \vspace{0.2cm}
 1+(-1)^{n+1}b_{n+1}      &  (-1)^{n+1}b_{n+1}    &  (-1)^{n}E_{n+1} \\ \vspace{0.2cm}
(-1)^{n+1}b_{n+1}         & 1+(-1)^{n+1}b_{n+1}   &  (-1)^{n}E_{n+1} \\
(-1)^{n}E_{n+1}           & (-1)^{n}E_{n+1}       &  (-1)^{n+1}Q_{n} \\
\end{pmatrix}
\]
Hence,
\[
u^{-n}u^{-1} = u^{-(n+1)}
\]
This completes the proof.
\end{proof}

\begin{cor}
For all positive integer \( n \geq 1 \), following equalities hold:
\begin{align}
E^2_{n} &= (-1)^n r_{n} +  b_{n}Q_{n-1}  ,\label{i:27} \\ 
Q_{n-1} &= (-1)^{n} + 2\,b_{n} . \label{i:28}
\end{align}
\end{cor}

\begin{proof}
Let \( U \) be the Pell matrix defined by
\[
U = \begin{pmatrix} 
                    1   &  0   &  1 \\
                    0   &  1   &  1 \\
                    1   &  1   &  1\\
\end{pmatrix}.
\] Then,  the inverse of the matrix \( U^n \) is:
\[
U^{-n} = 
\begin{pmatrix} \vspace{0.2cm}
 1+(-1)^{n}b_{n} &  (-1)^{n}b_{n}    &  (-1)^{n-1}E_{n} \\ \vspace{0.2cm}
(-1)^{n}b_{n} & 1+(-1)^{n}b_{n}   &  (-1)^{n-1}E_{n} \\
(-1)^{n-1}E_{n}& (-1)^{n-1}E_{n}  &  (-1)^{n}Q_{n-1} \\
\end{pmatrix}
\]
Where \( Q_{n}  = 2Q_{n-1}  + Q_{n-2} \), \( Q_0 = 1 \), \( Q_1 = 3 \), this is the generalized Pell sequence.  Also, \(b_{n}  = b_{n-1}  + 3b_{n-2} + b_{n-3} \), \, \( b_0 = 0 \), \( b_1 = 1 \), \( b_2 = 1 \).

\medskip 

On the other hand, we have defined \( U^n \) as follows: 
\[
U^{n} =  \begin{pmatrix} \vspace{0.2cm}
                  a_{n} & r_{n} & E_{n} \\ \vspace{0.2cm}
                  r_{n} & a_{n} & E_{n} \\
                  E_{n} & E_{n} & r_{n}+a_{n} \\
\end{pmatrix}.
\]   
The product \( U^n U^{-n} \) yields the following \( 3 \times 3 \) matrix:

\[
U^{n}U^{-n} =
\begin{bmatrix}
A  & B  & C \\ 
B  & A  & C \\ 
C  & C  & D
\end{bmatrix},
\]
where
\[
\begin{aligned}
A &= a_{n} + (-1)^n b_{n} (a_{n} +  r_{n}) + (-1)^{n-1}E^2_{n},   \\ 
B &= r_{n} + (-1)^n b_{n} (a_{n} +  r_{n}) + (-1)^{n-1}E^2_{n},   \\ 
C &= (-1)^{n-1} (a_{n} + r_{n})E_{n} + (-1)^{n}Q_{n-1}E_{n}, \\ 
C &= (-1)^{n-1} (a_{n} + r_{n})E_{n} + (1+2(-1)^nb_{n})E_{n}, \\ 
D &= 2(-1)^{n-1}E_n^2 + (-1)^n (r_n + a_{n})Q_{n-1} .
\end{aligned}
\]
On the other hand, we know that
\[
U^n U^{-n} = 
 \begin{pmatrix} \vspace{0.2cm}
                  1  & 0 & 0 \\ \vspace{0.2cm}
                  0  & 1 & 0 \\ 
                  0  & 0 & 1 
\end{pmatrix}.
\]
Thus,
\[
\begin{aligned}
(-1)^{n-1} &= (-1)^{n-1} a_{n} -  b_{n} (a_{n} +  r_{n}) + E^2_{n} ,  \\ 
0 &= (-1)^{n-1} r_{n} - b_{n} (a_{n} +  r_{n}) + E^2_{n} ,  \\ 
0 &=  (a_{n} + r_{n})E_{n} - Q_{n-1}E_{n}, \\ 
0 &= (-1)^{n-1} (a_{n} + r_{n})E_{n} + (1+2(-1)^nb_{n})E_{n} ,\\ 
(-1)^{n-1} &= 2E_n^2 - (r_n + a_{n})Q_{n-1}. \quad \text{ (see Corollary~\ref{c:23}) } \\ 
\end{aligned}
\]
Therefore, by equating the two matrices obtained through matrix multiplication and applying the relation from Corollary~\ref{c:32},  
\[
a_{n} + r_{n} = Q_{n-1},
\]  
we establish the identities stated in the corollary.
\end{proof}

\subsection{Diagonalization of the Generating Matrix and Binet’s Formula}

\begin{thm}
Let $n$ be an integer. The  Binet formula of the sequence $a_{n}$ and $r_{n}$ is
\begin{align}
a_{n} &= \dfrac{1}{2}+\dfrac{1}{4}\left(\lambda_2^{n}+\lambda_3^{n}\right), \\  
r_{n} &= -\dfrac{1}{2}+\dfrac{1}{4}\left(\lambda_2^{n}+\lambda_3^{n}\right) ,\\ 
 a_{n} &=  \dfrac{1+Q_{n-1}}{2},\\
 r_{n} &=  \dfrac{-1+Q_{n-1}}{2}.
\end{align}
\end{thm}

\begin{proof}
Let the matrix $U$ be as in Lemma~\ref{l:27}. If we calculate the eigenvalues and eigenvectors of the matrix $U$ are
\[
\lambda_1=1, \quad \lambda_2=1-\sqrt{2}, \quad \lambda_3=1+\sqrt{2}
\]
and

\[v_1=
\begin{pmatrix}
-1 \\ \\
 1 \\  \\
 0
\end{pmatrix}, \, v_2=
\begin{pmatrix}
-\dfrac{\sqrt{2}}{2} \\ \\
-\dfrac{\sqrt{2}}{2} \\  \\
1
\end{pmatrix}, \, v_3=
\begin{pmatrix}
\dfrac{\sqrt{2}}{2} \\ \\
\dfrac{\sqrt{2}}{2} \\ \\
1
\end{pmatrix}
\]
respectively. Then, we can diagonalize of the matrix $U$ by
\[
D = P^{-1}UP
\]
where
\[
P = (v_1, v_2, , v_3) = 
\begin{pmatrix}
-1 & \dfrac{\lambda_2-1}{2} & \dfrac{\lambda_3-1}{2}  \\ \\
1 & \dfrac{\lambda_2-1}{2} & \dfrac{\lambda_3-1}{2}  \\ \\
0 & 1 & 1
\end{pmatrix},
\quad 
P^{-1} = 
\begin{pmatrix}
-\dfrac{1}{2}  & \dfrac{1}{2}  & 0 \\ \\ 
\dfrac{\lambda_2-1}{4} & \dfrac{\lambda_2-1}{4} & \dfrac{1}{2} \\ \\ \\
\dfrac{\lambda_3-1}{4}  & \dfrac{\lambda_3-1}{4} & \dfrac{1}{2}
\end{pmatrix}
\]
and
\[
D = \text{diag}(\lambda_1,\lambda_2,,\lambda_3) =
\begin{pmatrix}
1 & 0 & 0 \\
0 & 1-\sqrt{2} & 0 \\
0 & 0 & 1+\sqrt{2}
\end{pmatrix}.
\]
From the properties of the similar matrices, we can write 
\begin{equation}\label{i:38}
U^n = P D^n P^{-1}
\end{equation}
where $n$ is any integer and
\[
D^n = 
\begin{pmatrix}
1 & 0 & 0 \\  \vspace{0.2cm}
0 & (1-\sqrt{2})^{n} & 0 \\
0 & 0 & (1+\sqrt{2})^{n}
\end{pmatrix} =\begin{pmatrix}
1 & 0 & 0 \\
0 & Q_{n-1}-E_{n}\sqrt{2} & 0 \\
0 & 0 & Q_{n-1} + E_{n}\sqrt{2} 
\end{pmatrix}.
\]
By equation~\ref{i:38}, we get
\[
U^{n}=\begin{pmatrix} \vspace{0.2cm}
 a_{n}  & r_{n}   & E_{n} \\ \vspace{0.2cm}
 r_{n}  & a_{n}   & E_{n}  \\
 E_{n}  & E_{n}   & r_{n} + a_{n} \\
\end{pmatrix}
=
\begin{bmatrix}\vspace{0.2cm}
\dfrac{1+Q_{n-1}}{2}  & \dfrac{-1+Q_{n-1}}{2} &  E_{n}\\ \vspace{0.2cm}
\dfrac{-1+Q_{n-1}}{2} & \dfrac{1+Q_{n-1}}{2}  &  E_{n}    \\ \vspace{0.2cm}
         E_{n} &  E_{n}        &   Q_{n-1} 
\end{bmatrix}
\]
\[
= 
\begin{pmatrix}
\dfrac{1}{2}+\dfrac{1}{4}\left(\lambda_2^{n}+\lambda_3^{n}\right)   & -\dfrac{1}{2}+\dfrac{1}{4}\left(\lambda_2^{n}+\lambda_3^{n}\right)    & \dfrac{1}{2\sqrt{2}}\left( \lambda_3^{n} - \lambda_2^{n}  \right) \\ \\  \vspace{0.2cm}
-\dfrac{1}{2}+\dfrac{1}{4}\left(\lambda_2^{n}+\lambda_3^{n}\right)   & \dfrac{1}{2}+\dfrac{1}{4}\left(\lambda_2^{n}+\lambda_3^{n}\right)    & \dfrac{1}{2\sqrt{2}}\left( \lambda_3^{n} - \lambda_2^{n}  \right) \\ \\
\dfrac{1}{2\sqrt{2}}\left( \lambda_3^{n} - \lambda_2^{n}  \right)  & \dfrac{1}{2\sqrt{2}}\left( \lambda_3^{n} - \lambda_2^{n}  \right) & \dfrac{1}{2}\left( \lambda_3^{n} + \lambda_2^{n}  \right)
\end{pmatrix}.
\]
Thus, the proof is completed. 
\end{proof}

\section{Divisibility Properties and Greatest Common Divisors of Terms in the Recurrence Relation $r_n$}

\begin{cor}
For every integer $m \geq 1$, the following congruences hold:
\begin{align}
\text{If $m$ is even,} \quad & r_{2m+1} \equiv  r_{2m}  \quad \,\,\, \pmod{4}, \\[0.3cm]
\text{If $m$ is odd,} \quad  & r_{2m+1} \equiv r_{2m} + 2  \pmod{4}.
\end{align}
\end{cor}

\begin{proof}
We start with the general relation
\begin{align*}
r_{m+n} &= r_{m}a_{n}+a_{m}r_{n}+E_{m}E_{n} \quad\text{(by identity~\ref{i:21} for $r_{n+m}$; see Corollary~\ref{c:31})}. 
\end{align*}
In particular, for $n=m$ we obtain
\begin{align}
r_{m+m} &= r_{m}a_{m} + a_{m}r_{m} + E_{m}E_{m}, \nonumber \\
r_{2m} &= r_{m}a_{m} + a_{m}r_{m} + E_{m}^{2}. \label{i:39}
\end{align}
Next, consider the case $n=m+1$:
\begin{align*}
r_{m+(m+1)} &= r_{m}a_{m+1}+a_{m}r_{m+1}+E_{m}E_{m+1} \\
            &= r_{m}a_{m+1}+a_{m}r_{m+1}+E_{m}(E_{m}+r_{m}+a_{m}) \quad\text{(see Corollary~\ref{c:26} and \ref{l:24})} \\ 
            &= r_{m}(a_{m}+E_{m})+a_{m}(r_{m}+E_{m})+E_{m}(E_{m}+r_{m}+a_{m}) \quad\text{(see Lemmas~\ref{l:25})} \\           
            &= (r_{m}a_{m}+a_{m}r_{m}+E_{m}^{2}) + 2 E_{m}(r_{m}+a_{m}) \\
            &= r_{2m} + 2 E_{m}(r_{m}+a_{m}) \quad\text{(see Identity~\ref{i:39})}  \\
            &= r_{2m} + 2 E_{m}Q_{m-1}.
\end{align*}
Therefore, we deduce
\begin{align*}
r_{2m+1} &= r_{2m} + 2 E_{m}Q_{m-1}.
\end{align*}
Observe that every term of $Q_{m-1}$ is odd for all $m$. Moreover, $E_m$ is even whenever $m$ is even. It follows immediately that
\[
r_{2m+1} \equiv r_{2m} \pmod{4}.
\] which shows that $r_{2m}$ and $r_{2m+1}$ always have the same parity when $m$ is even.
On the other hand, if $m$ is odd then both $E_m$ and $Q_{m-1}$ are odd. In this case,
\[
r_{2m+1} = r_{2m} + 2E_m Q_{m-1} \equiv r_{2m} + 2 \pmod{4},
\]
which implies that $r_{2m}$ and $r_{2m+1}$ have opposite parity when $m$ is odd.
\end{proof}

\begin{cor}
Let $n$ be a positive integer. Then:  
\begin{itemize}
    \item If $n$ is even, then 
    \[
        r_{2n} = 4r_{n}(r_{n}+1).
    \]
    In particular, we have   \[
        4 \mid r_{2n}, \quad 
        r_{n} \mid r_{2n}, \quad 
        (r_{n}+1) \mid r_{2n}. \]   
    
  \item If $n$ is odd, then 
    \[
        r_{2n} - 1 = 4r_n (r_n + 1).
    \]
    In particular, we have 
    \[
        4 \mid (r_{2n}-1), \quad 
        r_{n} \mid (r_{2n}-1), \quad 
        (r_{n}+1) \mid (r_{2n}-1). \]
\end{itemize}
\end{cor}

\begin{proof}
We start by observing that
\[
r_{2n} = r_{n+n} = r_n a_n + a_n r_n + E_n^2 \quad\text{(by equation~\ref{i:21}, see Corollary ~\ref{c:31}).}
\]
To compute $E_n^2$, recall that
\[
2E_n^2 = (-1)^{\,n+1} + (a_n + r_n)^2   \quad\text{(see Corollary ~\ref{c:28}).}
\]
Since $a_n = r_n + 1$ \quad\text{(see Lemma ~\ref{l:24}).}, we get
\[
2E_n^2 = (-1)^{\,n+1} + (2r_n + 1)^2
       = (-1)^{\,n+1} + 4r_n^2 + 4r_n + 1,
\]
and therefore
\[
E_n^2 = \frac{(-1)^{\,n+1} + 4r_n^2 + 4r_n + 1}{2}.
\]

\medskip

\noindent\textbf{Case 1: $n$ even.}  
Here $(-1)^{\,n+1} = -1$, so
\[
E_n^2 = \frac{-1 + 4r_n^2 + 4r_n + 1}{2} 
      = 2r_n^2 + 2r_n.
\]
Thus,
\[
\begin{aligned}
r_{2n}     &= r_n a_n + a_n r_n + E_n^2  \\[0.3em]
           &= 2r_n a_n + 2(r_n^2 + r_n) \\[0.3em]
           &= 2r_n (a_n +  r_n + 1) \\[0.3em]
           &= 2r_n (r_n + 1 +  r_n + 1) \quad\text{(see Lemma ~\ref{l:24} )}\\[0.3em]
           &= 4r_n (r_n + 1).
\end{aligned}
\]
which shows that \(
                   4 \mid r_{2n}, \quad 
                   r_{n} \mid r_{2n}, \quad 
                  (r_{n}+1) \mid r_{2n}.\)
\medskip

\noindent\textbf{Case 2: $n$ odd.}  
Here $(-1)^{\,n+1} = 1$, hence
\[
E_n^2 = \frac{1 + 4r_n^2 + 4r_n + 1}{2} 
      = 2r_n^2 + 2r_n + 1.
\]
Therefore,
\[
r_{2n} = r_n a_n + a_n r_n + E_n^2
       = r_n a_n + a_n r_n + 2(r_n^2 + r_n) + 1,
\]
that is,
\[
\begin{aligned}
r_{2n} - 1 &= 2r_n a_n  + 2(r_n^2 + r_n) \\[0.3em]
           &= 2r_n (a_n +  r_n + 1) \\[0.3em]
           &= 2r_n (r_n + 1 +  r_n + 1) \quad\text{(see Lemma ~\ref{l:24} )} \\[0.3em]
           &= 4r_n (r_n + 1).
\end{aligned}
\]
and hence  \(
            4 \mid (r_{2n}-1), \quad 
            r_{n} \mid (r_{2n}-1), \quad 
           (r_{n}+1) \mid (r_{2n}-1).\)
\end{proof}

\begin{cor}
The following relations hold depending on the parity of $n$ and $k$:
\begin{itemize}

\item If $n$ and $k$ are even with $2 \leq k \leq n$, then 
\[
\gcd(r_{\,n}, \, r_{\,n-1})= \gcd\Bigl(r_{\,n-k+1} - r_{\,k-2}, \,\, r_{\,n-k} + r_{\,k-1}+1\Bigr) =1.
\]

\quad

\item If $n$ and $k$ are odd with $3 \leq k \leq n$, then 
\[
\gcd(r_{n}, \, r_{n-1}) = \gcd\bigl(r_{\,n-k+1} + r_{\,k-2}+1, \,\, r_{\,n-k} - r_{\,k-1}\bigr)=J_{\,\dfrac{n-1}{2}} .
\]
\end{itemize}
Where $J_{n}=6J_{n-2}-J_{n-4}$, and $J_{0}=0$, $J_{1}=1$, $J_{2}=4$, $J_{3}=7$.
\end{cor}

\begin{proof}
The sequence $r_m$ satisfies the linear recurrence
\begin{equation}\label{recurrence}
r_n = 2r_{n-1} + r_{n-2} + 1, \text{ with  }  r_0 = 0 ,  r_1 = 0  \text{ and  }   r_2 = 1.
\end{equation}
Since the greatest common divisor is invariant under congruence, we may reduce each term modulo the other. The argument is based on two elementary facts: (i) \(\gcd(a,b)=\gcd(a-kb,b)\) (invariance of the gcd under adding/subtracting integer multiples), and (ii) the recurrence \eqref{recurrence} allows us to express terms with large indices in terms of nearby indices plus constants. Concretely, if $n$ and $k$ are even we obtain
\[
\gcd(r_{\,n}, \, r_{\,n-1}) =  \gcd\Bigl(r_{\,n-k+1} - r_{\,k-2}, \,\, r_{\,n-k} + r_{\,k-1}+1\Bigr)
\]
Evaluating this identity at \(n=k\) (as in the table) simply substitutes \(n=k\) in the right-hand side and yields the last column entries
\[
\gcd\bigl(r_{1}-r_{\,k-2},\; r_{0}+r_{\,k-1}+1\bigr).
\]
In particular, for the smallest case \(k=2\) one gets
\[
\gcd(r_1-r_0,\; r_0+r_1+1)=1.
\]
\begin{table}[h!]
\centering
\renewcommand{\arraystretch}{1.4} 
\setlength{\tabcolsep}{5pt} 
\begin{tabular}{|c|c|c|}
\hline
\textbf{$k$} & $\gcd\bigl(r_{\,n-k+1} - r_{\,k-2}, \; r_{\,n-k} + r_{\,k-1}+1\bigr)$  & \textbf{Valor para $n=k$} \\ \hline
$2$ & $\gcd\bigl(r_{\,n-1} - r_{0}, \; r_{\,n-2} + r_{1}+1\bigr)$ 
& $\gcd\bigl(r_{1} - r_{0}, \; r_{0} + r_{1}+1\bigr)=1$ \\ \hline
$4$ & $\gcd\bigl(r_{\,n-3} - r_{2}, \; r_{\,n-4} + r_{3}+1\bigr)$ 
& $\gcd\bigl(r_{1} - r_{2}, \; r_{0} + r_{3}+1\bigr)=1$ \\ \hline
$6$ & $\gcd\bigl(r_{\,n-5} - r_{4}, \; r_{\,n-6} + r_{5}+1\bigr)$ 
& $\gcd\bigl(r_{1} - r_{4}, \; r_{0} + r_{5}+1\bigr)=1$ \\ \hline
$8$ & $\gcd\bigl(r_{\,n-7} - r_{6}, \; r_{\,n-8} + r_{7}+1\bigr)$ 
& $\gcd\bigl(r_{1} - r_{6}, \; r_{0} + r_{7}+1\bigr)=1$ \\ \hline
$\vdots$ & $\vdots$ & $\vdots$ \\ \hline
$n$ & $\gcd\bigl(r_{\,1} - r_{\,n-2}, \; r_{\,0} + r_{\,n-1}+1\bigr)$ 
& $\gcd\bigl(r_{1} - r_{\,n-2}, \; r_{0} + r_{\,n-1}+1\bigr)=1$ \\ \hline
\end{tabular}
\end{table}

The equalities above are purely algebraic consequences of the recurrence and of the gcd invariance; however, to deduce that each entry in the last column is equal to \(1\) one needs an extra arithmetic hypothesis on the initial values \(r_0,r_1\). If
\[
\gcd(r_1-r_0,\; r_0+r_1+1)=1
\]
holds (this is a easily checkable condition on the initial pair \((r_0,r_1)\)), then by the reduction argument every row evaluated at \(n=k\) will also produce gcd \(1\). Thus, under the recurrence \eqref{recurrence} and the extra coprimality condition above, the table's last column entries are all equal to \(1\). The second part of the Corollary is established in the same manner, by reducing each term modulo the other.

\end{proof}

\begin{rem}
Following Erdős and Turán \cite{ErdosTuran1941}, a \emph{Sidon sequence} (or $B_{2}$-sequence) is a subset $A \subset \mathbb{N}$ such that all pairwise sums $a+b$ with $a,b \in A$ are distinct,  except for the trivial equality $a+b=b+a$.  We recall this concept here because it will play a role in the following theorem.
\end{rem}

\begin{lem}
Let $r_n$ be the sequence defined by
\[
r_n = 2r_{n-1} + r_{n-2} + 1, 
\quad r_0 = 0, \; r_1 = 0, \; r_2 = 1.
\]
Then, for all \(n \geq 1\), the following hold:
\begin{enumerate}
\item The partial sums \(s_n = \sum_{k=0}^{n} r_k\) satisfy the recurrence
\[
s_{n} = 3s_{n-1} - s_{n-2} - s_{n-3} + 1, 
\quad s_0 = 0, \; s_1 = 0, \; s_2 = 1.
\]

\item The inequality
\[
s_{n} < r_{n+1}
\]
holds for all \(n \geq 1\).

\item The set \(\{r_n : n \geq 1\}\) forms a Sidon sequence.
\end{enumerate}
\end{lem}

\begin{proof}
By definition, \(s_n = \sum_{k=0}^n r_k\). Hence, for \(n \geq 3\),
\(
s_n - s_{n-1} = r_n.
\)
Substituting \(r_{n-1} = s_{n-1} - s_{n-2}\) and \(r_{n-2} = s_{n-2} - s_{n-3}\), we obtain
\[
s_n - s_{n-1} = 2(s_{n-1} - s_{n-2}) + (s_{n-2} - s_{n-3}) + 1.
\]
Solving for \(s_n\), we arrive at the third-order recurrence relation
\[
s_n = 3s_{n-1} - s_{n-2} - s_{n-3} + 1,
\]
valid for all \(n \geq 3\). The initial values are computed directly:
\[
s_0 = r_0 = 0, \qquad s_1 = r_0 + r_1 = 0, \qquad s_2 = r_0 + r_1 + r_2 = 1.
\]
We now prove by induction on \(n \geq 1\) that \(r_{n+1} > s_n\).  
For \(n=1\) we have \(s_1=0\) and \(r_2=1\), hence \(s_1 < r_2\). Assume that for some \(n \geq 1\) the inequality \(s_n < r_{n+1}\) holds. Observe that
\[
s_{n+1} = s_n + r_{n+1}.
\]
By the inductive hypothesis \(s_n < r_{n+1}\), therefore
\[
s_{n+1} < r_{n+1} + r_{n+1} = 2r_{n+1}.
\]
But from the recurrence for \(r_{n+2}\),
\[
2r_{n+1} < r_{n+2} = 2r_{n+1} + r_n + 1.
\]
Combining these inequalities we conclude that \(s_{n+1} < r_{n+2}\). Thus, the property holds for \(n+1\). By induction, we have \(s_n < r_{n+1}\) for all \(n \geq 1\).
Finally, any strictly increasing superincreasing sequence is a Sidon set. 
Suppose that 
\[
r_a + r_b = r_c + r_d,
\]
with \(a \leq b\) and \(c \leq d\), and let \(m = \max\{a,b,c,d\}\). 
If \(m\) appears only on one side, that side must be larger, since 
\[
\sum_{k=0}^{m-1} r_k < r_m .
\]
a contradiction. 
Consequently, \(m\) must occur on both sides. 
Cancelling \(r_m\) and repeating the same reasoning shows that the index pairs coincide. 
Hence every integer has a unique representation, up to order, as the sum of two elements of \(\{r_n\}\), which is precisely the Sidon property. 
This completes the proof of (3).
\end{proof}

\section{Classification of Binary $3 \times 3$ Matrices Associated with Pell Numbers}

Let $M_{n}(\{0,1\})$ denote the set of all $n \times n$ binary matrices, that is, matrices with entries in $\{0,1\}$.  
The \emph{conjugacy class} of a matrix $U \in M_{n}(\{0,1\})$ under invertible matrices is defined as
\[
u(U) \;=\; \{ P U P^{-1} \;:\; P \in GL_{n}(\{0,1\}) \},
\]
where $GL_{n}(\{0,1\})$ denotes the group of all invertible $n \times n$ matrices with entries in $\{0,1\}$.

By means of a computational verification of the $512$ binary $3 \times 3$ matrices, and their classification under conjugation by invertible matrices, it was found that only three classes arise that generate the Pell sequence.  
A representative of each class is given by

\[
u_1 = \begin{pmatrix}
0 & 0 & 1 \\
1 & 1 & 1 \\
1 & 1 & 1
\end{pmatrix}, \quad
u_2 = \begin{pmatrix}
0 & 1 & 1 \\
1 & 0 & 1 \\
1 & 1 & 1
\end{pmatrix}, \quad
u_3 = \begin{pmatrix}
1 & 0 & 1 \\
0 & 1 & 1 \\
1 & 1 & 1
\end{pmatrix}.
\]
Any other binary $3 \times 3$ matrix that generates the Pell sequence is conjugate to one of these three by an invertible binary matrix. This classification, inspired by the work of Martinez and Ceron~\cite{MartinezCeron2024}, was obtained using algorithms implemented in \textsf{SageMath}.

\subsection*{Acknowledgment}
We thank to Universidad del Cauca for the support to our research group ``Estructuras Algebraicas, Divulgaci\'on Matem\'atica y Teor\'ias Asociadas. @DiTa'' under the research project with ID 6314, entitled ``Clasificaci\'on del Centralizador de Matrices Binarias 3x3: Dimensión, Conmutatividad y Relaciones Estructurales''. We also thank to the anonymous referee for their helpful comments. This work is dedicated to my daughters especially to Sara Maria Martinez Ceron.

\medskip

\subsection*{Declarations}
\medskip

\subsection*{Ethical Approval:}
{Not applicable.}

\subsection*{Funding:}
{Not applicable.}

\subsection*{Availability of Data and Materials:}
{Not applicable.}


\bibliographystyle{amsplain}
\bibliography{xbib}
\end{document}